\documentclass[12pt]{amsart}
\usepackage[utf8]{inputenc}
\usepackage[T1]{fontenc}
\usepackage{amsmath,amsfonts,amssymb,amstext,amsthm,bm}
\usepackage{graphicx}
\usepackage{mathrsfs}
\usepackage{hyperref}
\usepackage{color}
\usepackage[dvipsnames]{xcolor}
\usepackage{comment}
\usepackage{stackrel}
\usepackage{appendix}
\usepackage{enumitem}
\usepackage{tikz}

\newcommand{\circo}{~\raisebox{1pt}{\tikz \draw[line width=0.6pt] circle(1.1pt);}~}

\newtheorem{thm}{Theorem}[section]
\newtheorem{defi2}[thm]{Definition}
\newenvironment{defi}
{\begin{defi2}\rm}{\end{defi2}}
\newtheorem{pro}[thm]{Proposition}
\newtheorem{cor}[thm]{Corollary}
\newtheorem{lem}[thm]{Lemma}
\newtheorem{rem2}[thm]{Remark}
\newenvironment{rem}
{\begin{rem2}\rm}{\end{rem2}}
\newtheorem{example2}[thm]{Example}
\newenvironment{example}
{\begin{example2}\rm}{\end{example2}}
\numberwithin{equation}{section}

\newcommand{\R}{\mathbb{R}}
\newcommand{\N}{\mathbb{N}}

\newcommand{\D}{\displaystyle}

\def\qq#1{\qquad \mbox{#1}\quad}

\newcommand{\al}{\alpha}
\newcommand{\be}{\beta}

\newcommand{\e}{\varepsilon}

\newcommand{\ga}{\gamma}

\newcommand{\la}{\lambda}

\newcommand{\na}{\nabla}

\newcommand{\Om}{\Omega}
\newcommand{\Omb}{\overline{\Om}}
\newcommand{\p}{\partial}
\newcommand{\s}{\sigma}
\newcommand{\te}{\theta}

\newcommand{\hookto}{\hookrightarrow}

\usepackage{geometry}
\title[Elliptic systems with superlinear terms on the critical hyperbola]{Positive solutions of elliptic systems with superlinear terms on the critical hyperbola}

\author{Mabel Cuesta}
\address{Department of Mathematics, Université du Littoral Côte d'Opale (ULCO), Laboratoire de Mathématiques Pures et Appliqu\'ees Joseph Liouville (LMPA),		62100 Calais   (France)}
\email{\tt mabel.Cuesta@univ-littoral.fr}
\author{Rosa Pardo}
\address{Departamento de An\'alisis Matem\'atico y Matem\'atica Aplicada,  Universidad Complutense de Madrid, 28040{\textbf -}Madrid, Spain}
\email{\tt rpardo@ucm.es}
\author{Angela Pistoia}
\address{Dipartimento SBAI,	Sapienza Università di Roma, via Antonio Scarpa 16, 00161 Roma, Italy}
\email{\tt angela.pistoia@uniroma1.it}

\thanks{The second author is supported by grants PID2019-103860GB-I00, and PID2022-137074NB-I00,  MICINN,  Spain, and by UCM-BSCH, Spain, GR58/08, Grupo 920894. The third author is partially supported by GNAMPA-INdAM.}

\begin{document}
\maketitle
\begin{abstract}
We consider a slightly subcritical elliptic system with Dirichlet boundary conditions and a non-power nonlinearity in a bounded smooth domain.  For this problem, standard compact embeddings cannot be used to guarantee the existence of solutions as in the case of power-type nonlinearities.  Instead, we use the dual method on Orlicz spaces, showing that our problem possesses a mountain pass type solution. \\

\medskip

\end{abstract}

\section{Introduction}
Let us consider the system
\begin{equation}\label{s}
\left\{\begin{aligned}&-\Delta v= \frac{u^{p}}{\big(\ln (e+u)\big)^\alpha}\ \hbox{in}\ \Om\\
&-\Delta u= \frac{v^{q}}{\big(\ln (e+v)\big)^\beta}\ \hbox{in}\ \Om\\
& u>0,\, v>0\ \hbox{in}\ \Om\\
&u=v=0\ \hbox{on}\ \partial\Om\end{aligned}\right.
\end{equation}
where $\Om\subset\mathbb R^N$, $N\ge3$, is a  bounded domain of class $C^2,$
$p,q>0$, $\alpha\le p,\ \beta\le q,$ and $(p,q)$ may belong to the critical hyperbola; specifically, either
\begin{equation}\label{h:sub}
1>\frac1{p+1}+\frac1{q+1}>\frac{N-2}N,
\end{equation}
or
\begin{equation}\label{h}
\frac1{p+1}+\frac1{q+1}=\frac{N-2}N,\qquad  \frac{\al}{p+1}+ \frac{\be}{q+1}>0.
\end{equation}
We want to find a  solution $(u,v)$ of \eqref{s}, positive in both components.\\

Problems of type  \eqref{s} has been considered by several authors, we refer to \cite{Bonheure_dosSantos_Tavares, Clement_dePagter_Sweers_deThelin, deFigueiredo_doO_Ruf_2005_Orlicz, Mavinga_Pardo_JMAA}. In 
\cite{Bonheure_dosSantos_Tavares} the authors study the case \eqref{h:sub} and $\al=\be=0$.
Whereas in \cite[Theorem 2.7]{Clement_dePagter_Sweers_deThelin} the authors study related nonlinearities  when the pair of exponents $(p,q)$ lies below the critical Sobolev hyperbola \eqref{h:sub}, using  variational approaches. 
In \cite[Theorem 1.3]{deFigueiredo_doO_Ruf_2005_Orlicz}, using an Orlicz-space approach, the authors study the existence of solutions of \eqref{s}  when the pairs of exponents $(p,q)$, $(\al,\be)$ satisfy \eqref{h}, in particular $(p,q)$ lies on the critical Sobolev hyperbola, with $p,q>1$; they do not cover the case $p,q>0$. 
While in \cite{Mavinga_Pardo_JMAA} 
the authors establish a-priori $L^\infty$ bounds when the pair of exponents $(p,q)$ lies on the critical Sobolev hyperbola \eqref{h}, $1< p,q<\infty$, and $\displaystyle\alpha,\, \beta> 2/(N-2).$ 
Concerning one single equation, we can mention \cite{Castro_Pardo_RMC} establishing $L^\infty$ a priori bounds when  $\alpha> 2/(N-2)$, \cite{Clapp-Pardo-Pistoia-Saldana} analyzing the asymptotic behavior  of the solutions as $\al\to 0$, and \cite{Cuesta_Pardo_MilanJM}
including a changing sign weight.

\medskip

Our main result is the following:
\begin{thm}\label{main}
Assume $p,q>0$, $\alpha\le p,\ \beta\le q,$ and either $(p,q)$ satisfy \eqref{h:sub}  or  $(p,q)$, $(\al,\be)$ satisfy \eqref{h}. 
Then Problem \eqref{s} possesses  a solution of mountain pass type.
\end{thm}

The proof relies on  the dual method which allows to reduce the existence of solutions to the problem \eqref{s} to  finding critical points of an energy functional defined on a suitable Orlicz space, whose accurate choice  is the main novelty of the present paper.\\
\begin{rem2} We would like to  comment on condition \eqref{h}. At this aim we 
 point out that Theorem \ref{main} also applies to the case of the single equation, namely $p=q=\frac{N+2}{N-2}$
and $\alpha=\beta>0$. In particular,   the problem
\begin{equation}\label{s-e}
\left\{\begin{aligned}&-\Delta u= \frac{u^{p}}{\big(\ln (e+u)\big)^\alpha}\ \hbox{in}\ \Om\\
& u>0\ \hbox{in}\ \Om\\
&u= 0\ \hbox{on}\ \partial\Om\end{aligned}\right.
\end{equation}
has always a positive solution. On the other hand, it is known  that Pohozaev's identity \cite{Pohozaev} ensures the non-existence of positive solutions on star-shaped domains to equation \eqref{s-e} whenever $\alpha\le0,$ see Remark \eqref{rem}.
\\
In a similar way, it could be interesting to prove  the {\em criticality} of our  condition \ref{h}. Indeed, a question naturally arises:
{\em if $p$ and $q$ lie on the critical hyperbola and 
$$ \frac{\al}{p+1}+ \frac{\be}{q+1}\le0$$
does the system \eqref{s} have any positive solutions on a star-shaped domain?}
In Theorem \ref{th:non:ex} we give a partial answer when both $\alpha$ and $\beta$ are not positive. The general case remains open.
\end{rem2}

The paper is organized as follows. In Section \ref{sec:dual} we describe the dual method which allows us to consider $p,\ q>0$. Roughly speaking, this method consists in taking the inverse of the Laplace operator, and defining  the inverse of the nonlinearities. In Section \ref{sec:proof} we prove our main result, that the associated energy functional, defined in a suitable product of Orlicz spaces, has a mountain pass geometry.
Section \ref{sec:crit} is devoted to analyze the  criticality of the condition \ref{h}.
Finally, Section \ref{sec:orl} summarize the theory of Orlicz spaces needed for our purposes.

\section{The dual method}
\label{sec:dual}
\subsection{The variational formulation}
Given any $s\in (1,\infty)$, we introduce 
the operator $\mathtt K:=(-\Delta)^{-1}: L^{s}(\Om)\to W^{2,s}(\Om)\cap W_0^{1,s}(\Om)$ which is defined as
$$\mathtt K (f):=u\qq{if and only if}  -\Delta u=f\ \hbox{in}\ \Om,\ u=0\ \hbox{on}\ \partial\Om.$$
Next, we set
\begin{equation}
\label{def:a:b}
a(t):=\frac{t^p}{\big(\ln (e+t)\big)^\alpha}, \qq{and}  b(t):= \frac{t^q}{\big(\ln (e+t)\big)^\beta}\ \hbox{ with}\ t\geq 0,
\end{equation}
and consider their primitive functions:
\begin{equation}
\label{def:A:B}
A(t):=\int\limits_0^t a(s)\, ds \qq{and}  B(t):=\int\limits_0^t b(s)\, ds  \qq{for all} t\geq 0.
\end{equation}

If $\alpha\le p$ and $\beta\le q$, the functions $a$, $b$ are strictly increasing and so invertible in $[0,+\infty)$. 
Let us denote by 
\begin{equation}
\label{def:tilde:a:b}
\tilde{a}:=a^{-1} \qq{and} \tilde{b}:=b^{-1}
\quad\text{the  inverse functions.} 
\end{equation}
If we set
$f=a(u)$  and $g=b(v)$ or equivalently $u=\tilde{a} (f)$ and $v=\tilde{b}(g),$ then problem \eqref{s} can be rewritten as
\begin{equation}\label{s2}\left\{\begin{aligned}&\tilde{b} (g)= \mathtt K (f)\ \hbox{in}\ \Om\\
&\tilde{a}( f)= \mathtt K(g)\ \hbox{in}\ \Om\\
& f>0,\, g>0 \ \hbox{in}\ \Om .
&\end{aligned}\right.\end{equation}
Let us introduce their primitive functions:
\begin{equation}
\label{def:tilde:A:B}
\widetilde A(t):=\int\limits_0^t \tilde{a}(s)\, ds \qq{and} \widetilde B(t):=\int\limits_0^t \tilde{b}(s)\, ds,
\end{equation}
for all $t\geq 0$, and  denote  $L^{\tilde A}(\Om)$ and $L^{\tilde B}(\Om)$  the {\it Orlicz spaces}
associated to the functions ${\widetilde A}$ and ${\widetilde B}$ respectively, cf. definition \ref{LH} and  Remark \ref{tildeh:h-1}.  Observe that, since  $\tilde{a}$ and $\tilde{b}$ are continuous and increasing  the functions $\tilde A$ are $\tilde B$ are $ \mathcal N$- functions, cf. definition \ref{Nfunctions}. We endowed the space  $L^{\tilde A}(\Om)$ (resp. $L^{\tilde B}(\Om)$) with either the {\it Luxembourg norm $\|\cdot\|_{(\tilde A)} $} or the {\it Orlicz norm } $\|\cdot \|_{\tilde A}$ (resp. $\|\cdot \|_{(\tilde B )}$ and $\|\cdot \|_{\tilde B}$), see definitions \eqref{def:L-norm} and \eqref{onorm} for those norms. From Lemma \ref{Ban} and Proposition \ref{proLH}(i), the spaces  $L^{\tilde A}(\Om)$ and  $L^{\tilde B}(\Om)$  are complete  with both the Orlicz or the Luxemburg norm.

\medskip
Problem \eqref{s2} has a {\it variational structure}.  Indeed, let us  denote by $X$ the Banach space 
$$
X=L^{\tilde{A}}(\Om)\times L^{\tilde{B}}(\Om)
$$
endowed with the norm $\|(f,g)\|_X: =\|f\|_{\tilde A}+ \|g\|_{\tilde B}$.

By the regularity result
Theorem \ref{regu},  the operator  ${\mathtt K}$ sends  $ L^{\tilde A}(\Om)$ into $ W_0^{1, {\widetilde A}}(\Om)\cap W^{2,{\tilde A}}(\Om)$, cf. Definition  \ref{def:O-S} of the {\it Orlicz-Sobolev spaces} $W_0^{1,H} (\Omega)$ and  $W^{2,H} (\Omega)$, when $H$ is equal to either $\tilde A$ or $\tilde B$.
Thus for any  $(f,g)\in X$, $({\mathtt K}(f),{\mathtt K}(g))\in W^{2,{\tilde A}}(\Om)\times W^{2,{\tilde B}}(\Om).$ 
By a {\it solution} $(u,v)$ of \eqref{s} we mean a pair $(u,v)\in \big(W_0^{1, {\widetilde A}}(\Om)\cap W^{2,{\tilde A}}(\Om)\big)\times\big( W_0^{1, {\widetilde B}}(\Om)\cap W^{2,{\tilde B}}(\Om)\big)$ satisfying \eqref{s} in the weak sense.
\medskip

Moreover, for $(p,q),\ (\al,\be)$ satisfying either \eqref{h:sub},
or \eqref{h}, 
we have that the embeddings  $W^{2,{\tilde A}}(\Om)\hookrightarrow L^{B}(\Om)$ and   $W^{2,{\tilde B}}(\Om)\hookrightarrow L^{A}(\Om)$, 
are continuous and compact, see Lemma \ref{lem:comp:emb}.
Consequently, $f\,{\mathtt K}(g)  \in L^1(\Om)$, $g\,{\mathtt K} (f) \in L^1(\Om)$,   see  the 2nd Holder's inequality \eqref{2HIn} in   Proposition \ref{proLH}. Furthermore, it follows from the definition of $\mathtt K$, that
\begin{equation}\label{K:na}
\int\limits_\Om f\,{\mathtt K}(g) \, dx=\int_\Om \nabla\, {\mathtt K}(f) \cdot \nabla\, {\mathtt K}(g) \, dx =\int\limits_\Om g\,{\mathtt K} (f) \, dx. 
\end{equation}

For convenience, let us extend $a$, $\tilde{a}$, $b$, $\tilde{b}$  in the whole $\mathbb{R}$ as  odd functions, and extend $A$, $\widetilde A$, $B$ and $\widetilde B$ as even functions.   We next  introduce the $C^1$-functional  $J:X\to \mathbb{R}$ defined by
\begin{equation}\label{phi}
J(f,g):=\int\limits_\Om \widetilde A(f)\, dx +  \int\limits_\Om\widetilde B(g)\, dx-\frac{1}{2}\int\limits_\Om \big(f\,\mathtt K (g)+g\,\mathtt K (f)\big)\, dx \end{equation}
whose derivative at $(f,g)\in X$ is equal to 
$$
J'(f,g)[\psi_1,\psi_2]=
\left(\int\limits_\Om \tilde{a}(f)\psi_1 \, dx +\int\limits_\Om\widetilde  b(g)\psi_2 \, dx\right) -\left(\int\limits_\Om \mathtt \psi_1 K (g)+\int\limits_\Om \psi_2\,\mathtt K (f) \, dx\right) 
,
$$
for all  $(\psi_1,\psi_2)\in X.$ \\

Trivially, the equations \eqref{s2} are the Euler-Lagrange equations associated to the action functional $J$.

\subsection{On the nonlinearities}
\begin{lem} \label{lem:tilde:aA}
{\rm (i)} There exists two constants $c_1,\ c_2>0$, only dependent of $p,$ and $\al,$ 
such that
\begin{equation}\label{eq:tilde:a}
c_1 t^\frac{1}{p}\big(\ln (e+t)\big)^{\frac{\al}{p}}\le \tilde{a} (t)\le c_2 t^\frac{1}{p}\big(\ln (e+t)\big)^{\frac{\al}{p}} ,\qq{for all}t\ge 0.
\end{equation} 
{\rm (ii)} There exists two constants $C_1,\ C_2>0$ only dependent of $p,$ and $\al,$ such that
\begin{equation*}
C_1\,  t^{\frac{p +1}p}\big(\ln (e+t)\big)^\frac\alpha p\le \widetilde A(t)\le C_2\, t^{\frac{p +1}p}(\big(\ln (e+t)\big)^\frac\alpha p ,\qq{for all}  t\ge 0.
\end{equation*}   
\end{lem}

\begin{proof}[Proof of Lemma \ref{lem:tilde:aA}]
(i) 
By definition, 
\begin{equation*}
a\Big(t\big(\ln (e+t)\big)^{\frac{\al}{p}}\Big) =  t^p\big(\ln (e+t)\big)^{\al}\left[\ln
\left(e+t\big(\ln (e+t)\big)^{\frac{\al}{p}}\right)\right]^{-\al}.
\end{equation*}
Now, checking that
\begin{equation*}\label{lim1}
\lim_{t\to 0}\frac{\D\ln
\left(e+t\big(\ln (e+t)\big)^{\frac{\al}{p}}\right)}{\ln (e+t)}=1,\qquad
\lim_{t\to \infty}\frac{\D\ln \left(e+t\big(\ln (e+t)\big)^{\frac{\al}{p}}\right)}{\ln (e+t)}=1,
\end{equation*}
and since that quotient of logarithms is a continuous, non vanishing function for $t\ge0$,
we can conclude that there exists two constants $d_1,\ d_2>0$ such that 
$$
d_1 t^p\le a\Big(t\big(\ln (e+t)\big)^{\frac{\al}{p}}\Big) \le  d_2 t^p,\qq{for all}t\ge 0,
$$ 
and since $\tilde{a}$ is increasing,
$$
\tilde{a} ( d_1 t^p)\le t\big(\ln (e+t)\big)^{\frac{\al}{p}} \le  \tilde{a} ( d_2 t^p),\qq{for all}t\ge 0.
$$
Now, firstly denoting by $s=d_1 t^p$, secondly using that 
$$
\lim_{t\to 0}\frac{\D\ln\left(e+t^{1/p}\right)}{\ln (e+t)}=1,\qquad
\lim_{t\to \infty}\frac{\D\ln\left(e+t^{1/p}\right)}{\ln (e+t)}=1,
$$
and thirdly, denoting by $s=d_2 t^p$, we obtain \eqref{eq:tilde:a}.

\bigskip

{\rm (ii)} Given $\ga>0,$ and $\nu\ge -\ga$, set
\begin{equation*}
\s_{\ga,\nu}(t)=\s(t):=t^\ga\, \big(\ln (e+t)\big)^\nu, \qq{for} t>0,
\end{equation*}
then
\begin{equation*}
\s'_{\ga,\nu}(t)=\s'(t)
=t^{\ga -1}\, \big(\ln (e+t)\big)^\nu\left[\ga +\nu\frac{1}{\ln (e+t)}\,\frac{t}{(e+t)}
\right].
\end{equation*}
Observe that 
\begin{equation*} 
a(t)=\s_{p,-\al}(t),
\end{equation*}
and that 
\begin{equation}\label{a:ta'}
\lim_{t\to\infty}   \frac{ta'(t)}{a(t)}=p.
\end{equation} 
Moreover, if $\nu\ge -\ga,$ then  $\s$ is increasing. 

Set
$$
h(t):=\frac{1}{\ln (e+t)}\frac{t}{(e+t)}>0, \qq{for} t>0.
$$
Trivially,  $\lim_{t\to 0}h(t)=0,$ $\lim_{t\to \infty}h(t)=0, $ and there exists a unique 
$t^*>0$ such that 
$$
\ h(t)\le h(t^*)=\frac{e}{e+t^*}<1, \qquad \forall t\ge 0.
$$
Hence
\begin{equation}
\label{sigma'<<}
\left(\ga +\frac{|\nu|\,e}{e+t^*}\right)^{-1}\s'_{\ga ,\nu}(t)
\le  t^{\ga -1}\, \big(\ln (e+t)\big)^\nu  \le \left(\ga -\frac{|\nu|\,e}{e+t^*}\right)^{-1}\s'_{\ga ,\nu}(t),
\end{equation}
and observe that whenever $\ga>1$, $t^{\ga -1}\, \big(\ln (e+t)\big)^\nu= \s_{\ga -1,\nu}(t)$, so
\begin{equation*}
\left(\ga +\frac{|\nu|\,e}{e+t^*}\right)^{-1}\s'_{\ga ,\nu}(t)
\le   \s_{\ga -1,\nu}(t) \le \left(\ga -\frac{|\nu|\,e}{e+t^*}\right)^{-1}\s'_{\ga ,\nu}(t).
\end{equation*}

By definition of $\widetilde A$ (see \eqref{def:A:B}), using \eqref{eq:tilde:a}, and \eqref{sigma'<<}, 
we can write
\begin{equation*}
\widetilde{A}(t) =\int\limits_0^t \tilde{a}(s)\, ds \le c_2 \int\limits_0^t s^\frac{1}{p}\big(\ln (e+s)\big)^{\frac{\al}{p}}\, ds\le C_2 \s_{\frac{p+1}{p}, \frac{\al}{p}} (t)=C_2 t^\frac{p+1}{p}\, \big(\ln (e+t)\big)^{\frac{\al}{p}}.
\end{equation*}
Likewise is obtained the reverse inequality.
\end{proof}

\begin{rem}\label{rem:A}
Likewise, there exists two constants $C'_1,\ C'_2>0$ only dependent of $p,$ and $\al,$ such that
\begin{equation*}
C'_1\,  t^{p +1}\big(\ln (e+t)\big)^{-\al}\le  A(t)\le C'_2\, t^{p +1}(\big(\ln (e+t)\big)^{-\al} ,\qq{for all}  t\ge 0.
\end{equation*}  
In fact, from \eqref{sigma'<<}, for all  $t\ge 0$,
\begin{equation*}
\left(p +1 +\frac{|\al|\,e}{e+t^*}\right)^{-1}\,  t\,a(t)\le  A(t)\le \left(p +1 -\frac{|\al|\,e}{e+t^*}\right)^{-1}\, t\,a(t) .
\end{equation*} 
\end{rem}

\begin{lem} \label{lem:norm:tilde:A}
(i) Assume that $\al\ge 0$. Then, for all  $f\in L^{\widetilde A}(\Om),$  and for the constants  $C_1,\ C_2>0$ provided by Lemma \ref{lem:tilde:aA},  the following hold
\begin{equation}\label{eq:tilde:A:<1}
\int_\Om  {\widetilde A} (f)\, dx\ge \frac{C_1}{C_2}\,\,\|f\|_{(\tilde A)}^{\,\frac{p +1}p} ,\qq{whenever}\|f\|_{(\tilde A)}\ge 1,
\end{equation}

\begin{equation}\label{eq:tilde:A:>1}
\int_\Om  {\widetilde A} (f)\, dx\le \frac{C_2}{C_1}\,\,\,  \|f\|_{(\tilde A)}^{\,\frac{p +1}p} ,\qq{whenever} \|f\|_{(\tilde A)}\le 1,
\end{equation} 
and 
\begin{equation}\label{eq:tilde:A:p+1}
\int_\Om  {\widetilde A} (f)\, dx \ge  C_1\,  \|f\|_{\frac{p +1}p}^{\,\frac{p +1}p}.
\end{equation}

{\rm (ii)} Assume now that $\al<0$. Then, 
\begin{equation}\label{eq:tilde:A:<2}
\int_\Om  {\widetilde A} (f)\, dx\le \frac{C_2}{C_1}\|f\|_{(\tilde A)}^{\,\frac{p +1}p} ,\qq{whenever} \|f\|_{(\tilde A)}\ge 1,
\end{equation} 
\begin{equation}\label{eq:tilde:A:>2}
\int_\Om  {\widetilde A} (f)\, dx\ge \frac{C_1}{C_2}\,\,  \|f\|_{(\tilde A)}^{\,\frac{p +1}p} ,\qq{whenever} \|f\|_{(\tilde A)}\le 1,
\end{equation} 
and if $f\in L^\frac{p +1}p (\Om)$, then
\begin{equation*}
\int_\Om  {\widetilde A} (f)\, dx \le  C_2\,  \|f\|_{\frac{p +1}p}^{\,\frac{p +1}p}.
\end{equation*}

\end{lem}
\begin{rem} Likewise, for all  $f\in L^{A}(\Om),$  and for the constants  $C'_1,\ C'_2>0$ provided by Remark \ref{rem:A}, the following hold

{\rm (i)} Assume that $\al\ge 0$. Then, 
\begin{equation*}
\int_\Om  {A} (f)\, dx\le \frac{C'_2}{C'_1}\,\,\|f\|_{(A)}^{p +1} ,\qq{whenever}   \|f\|_{(A)}\ge 1,
\end{equation*} 
and 
\begin{equation*}
\int_\Om  {A} (f)\, dx\ge \frac{C'_1}{C'_2}\,\,\,  \|f\|_{(A)}^{p +1} ,\qq{whenever} \|f\|_{(A)}\le 1.
\end{equation*} 
and if $f\in L^{p +1} (\Om)$, then
\begin{equation}\label{eq:A:p+1:>}
\int_\Om   A (f)\, dx \le  C'_2\,  \|f\|_{p +1}^{\,p +1}.
\end{equation}

{\rm (ii)} Assume now that $\al<0$. Then
\begin{equation*}
\int_\Om  A (f)\, dx\ge \frac{C'_1}{C'_2}\,\,\|f\|_{(A)}^{p +1} ,\qq{whenever} \|f\|_{(A)}\ge 1,
\end{equation*}

\begin{equation}\label{eq:A:<2}
\int_\Om  A (f)\, dx\le \frac{C'_2}{C'_1}\,\,\,  \|f\|_{(A)}^{p +1} ,\qq{whenever} \|f\|_{(A)}\le 1.
\end{equation} 
and 
\begin{equation*}
\int_\Om  {A} (f)\, dx \ge  C'_1\,  \|f\|_{p +1}^{\,p +1}.
\end{equation*}
\end{rem}
\begin{proof}[Proof of Lemma \ref{lem:norm:tilde:A}]
Let 
\begin{equation*}
{\widetilde A_1}(s):=s^{\frac{p +1}p}\big(\ln (e+s)\big)^{\frac{\al}p}. 
\end{equation*}
Since Lemma \ref{lem:tilde:aA} 
$$
C_1\,  \int_\Om  {\widetilde A_1} (f)\, dx\le \int_\Om  {\widetilde A} (f)\, dx
\le C_2\,  \int_\Om  {\widetilde A_1} (f)\, dx.
$$

(i) Assume that $\al>0$. If $\|f\|_{(\tilde A)}\ge 1 $, then by Lemma  \ref{Ban}{\rm (iv)}
\begin{align*}
1&= \int_\Om  {\widetilde A} \left(\frac{|f|}{\|f\|_{(\tilde A)}}\right)\, dx    
\le C_2  \int_\Om \left(\frac{|f|}{\|f\|_{(\tilde A)}}\right)^{\frac{p +1}p}\big(\ln (e+|f|)\big)^{\frac{\al}p} \left(\frac{\ln \left(e+\frac{|f|}  {\|f\|_{(\tilde A)}}\right)}{\ln (e+|f|)}\right)^{{\frac{\al}p}}
\\
&\le   \frac{C_2}{\|f\|_{(\tilde A)}^{\,\frac{p +1}p}}\,\int_\Om |f|^{\frac{p +1}p}\big(\ln (e+|f|)\big)^{\frac{\al}p} =\frac{C_2}{\|f\|_{(\tilde A)}^{\,\frac{p +1}p}}\,  \int_\Om  {\widetilde A_1} (f)\, dx
\le \frac{C_2}{C_1}\,\,  \frac{1}{\|f\|_{(\tilde A)}^{\,\frac{p +1}p}}\, \int_\Om  {\widetilde A} (f)\, dx.
\end{align*}
Consequently, \eqref{eq:tilde:A:<1} holds. 

\medskip

Moreover, for all $f\in L^{\widetilde A}(\Om),$ with $\|f\|_{(\tilde A)}\le 1 $, 
\begin{align*}
1&=\int_\Om  {\widetilde A} \left(\frac{|f|}{\|f\|_{(\tilde A)}}\right)\, dx    \ge C_1  \int_\Om \left(\frac{|f|}{\|f\|_{(\tilde A)}}\right)^{\frac{p +1}p}\big(\ln (e+|f|)\big)^{\frac{\al}p} \left(\frac{\ln \left(e+\frac{|f|}  {\|f\|_{(\tilde A)}}\right)}{\ln (e+|f|)}\right)^{{\frac{\al}p}}\\
&\ge   \frac{C_1}{\|f\|_{(\tilde A)}^{\,\frac{p +1}p}}\,\int_\Om |f|^{\frac{p +1}p}\big(\ln (e+|f|)\big)^{\frac{\al}p} =\frac{C_1}{\|f\|_{(\tilde A)}^{\,\frac{p +1}p}}\,  \int_\Om  {\widetilde A_1} (f)\, dx
\ge \frac{C_1}{C_2}\,\,  \frac{1}{\|f\|_{(\tilde A)}^{\,\frac{p +1}p}}\, \int_\Om  {\widetilde A} (f)\, dx.
\end{align*}
So \eqref{eq:tilde:A:>1} holds.

Besides, for all $f\in L^{\widetilde A}(\Om),$
\begin{equation*}
\int_\Om  {\widetilde A} (f)\, dx \ge  C_1\,  \int_\Om  {\widetilde A_1} (f))\, dx = C_1 \int_\Om |f|^{\frac{p +1}p}\big(\ln (e+|f|)\big)^{\frac{\al}p} \,dx
\ge  C_1 \|f\|_{\frac{p +1}p}^{\,\frac{p +1}p}.
\end{equation*}

{\rm (ii)} Assume now that $\al<0$, and let us write now ${\widetilde A_2}(s):=s^{\frac{p +1}p}\big(\ln (e+s)\big)^{-\frac{|\al|}p}$. 
Assume now $\|f\|_{(\tilde A)}\le 1 $, then  
\begin{align*}
1&=\int_\Om  {\widetilde A} \left(\frac{|f|}{\|f\|_{(\tilde A)}}\right)\, dx    \le C_2  \int_\Om \left(\frac{|f|}{\|f\|_{(\tilde A)}}\right)^{\frac{p +1}p}\big(\ln (e+|f|)\big)^{-\frac{|\al|}p} \left(\frac{\ln (e+|f|)}{\ln \left(e+\frac{|f|}  {\|f\|_{(\tilde A)}}\right)}\right)^{{\frac{|\al|}p}}\\
&\le   \frac{C_2}{\|f\|_{(\tilde A)}^{\,\frac{p +1}p}}\,\int_\Om |f|^{\frac{p +1}p}\big(\ln (e+|f|)\big)^{-\frac{|\al|}p} =\frac{C_2}{\|f\|_{(\tilde A)}^{\,\frac{p +1}p}}\,  \int_\Om  {\widetilde A_1} (f)\, dx
\le \frac{C_2}{C_1}\,\,  \frac{1}{\|f\|_{(\tilde A)}^{\,\frac{p +1}p}}\, \int_\Om  {\widetilde A} (f)\, dx.
\end{align*}
So \eqref{eq:tilde:A:>2} holds.

\medskip

Moreover, for all $f\in L^{\widetilde A}(\Om),$ with $\|f\|_{(\tilde A)}\ge 1 $, 
\begin{align*}
1&= \int_\Om  {\widetilde A} \left(\frac{|f|}{\|f\|_{(\tilde A)}}\right)\, dx    \ge C_1  \int_\Om \left(\frac{|f|}{\|f\|_{(\tilde A)}}\right)^{\frac{p +1}p}\big(\ln (e+|f|)\big)^{-\frac{|\al|}p} \left(\frac{\ln (e+|f|)}{\ln \left(e+\frac{|f|}  {\|f\|_{(\tilde A)}}\right)}\right)^{{\frac{|\al|}p}}
\\
&\ge   \frac{C_1}{\|f\|_{(\tilde A)}^{\,\frac{p +1}p}}\,\int_\Om |f|^{\frac{p +1}p}\big(\ln (e+|f|)\big)^{-\frac{|\al|}p} =\frac{C_1}{\|f\|_{(\tilde A)}^{\,\frac{p +1}p}}\,  \int_\Om  {\widetilde A_1} (f)\, dx
\ge \frac{C_1}{C_2}\,  \frac{1}{\|f\|_{(\tilde A)}^{\,\frac{p +1}p}}\, \int_\Om  {\widetilde A} (f)\, dx.
\end{align*}
Consequently, \eqref{eq:tilde:A:<2} holds.\\

Finally,  Lemma \ref{lem:tilde:aA}(ii) ends the proof.
\end{proof}

\subsection{On the space $X=L^{\tilde A}(\Om)\times L^{\tilde B}(\Om)$}

\medskip

We  prove in this section that $X$ is reflexive.

To do so, in Proposition \ref{proLH}(v) is stated that we only need to check that the ${\mathcal N}$-functions satisfy the  $\Delta_2$-condition at infinity (see \eqref{d2infty} for a definition of  the  $\Delta_2$-condition at infinity).

\begin{pro}\label{ref}
$A$, $B$, ${\widetilde A}$, ${\widetilde B}$ satisfy the $\Delta_2$-condition, and
$X$ is a reflexive  Banach space. Moreover, the dual space 
\begin{equation}\label{dual:X}
X':=\left(L^{\tilde A}(\Om)\times L^{\tilde B}(\Om)\right)'= \left(L^{\tilde A}(\Om)\right)'\times \left(L^{\tilde B}(\Om)\right)'=
L^{ A}(\Om)\times L^{ B}(\Om)
\end{equation}
\end{pro}

\begin{proof}
From Lemma \ref{Ban} and Proposition \ref{proLH}(i), the spaces  $L^{\tilde A}(\Om)$ and  $L^{\tilde B}(\Om)$  are complete  with both the Orlicz or the Luxemburg norm. From  Proposition \ref{proLH}{\rm (v)} it is enough to check that $A$, $B$, ${\widetilde A}$, ${\widetilde B}$ satisfy the $\Delta_2$-condition.  To do that, we apply Lemma \ref{delta2bis} (iii), see condition \eqref{delta22}.

Indeed, notice that  by l'Hôpital rule  
$$
\lim_{t\to+\infty}\frac{ta(t)}{A(t)}=\lim_{t\to+\infty}\frac{ta'(t)}{a(t)} +1=p+1,
\qq{and}
\lim_{t\to+\infty}\frac{tb(t)}{B(t)}=q+1,
$$
see \eqref{a:ta'} in the proof of Lemma \ref{lem:tilde:aA}.
Now, changing the the variable $s=\tilde{a}(t)$, using Young equality\eqref{ye}, and l'Hôpital rule   
$$
\lim_{t\to+\infty}\frac{t\tilde{a}(t)}{\tilde{A}(t)}
=\lim_{s\to+\infty}\frac{sa(s)}{\tilde{A}(a(s))}=\lim_{s\to+\infty}\frac{sa(s)}{sa(s)-A(s)}
=\lim_{s\to+\infty}\frac{a(s)+sa'(s)}{sa'(s)}=\frac1{p}+1,
$$
see also \eqref{a:ta'} in the proof of Lemma \ref{lem:tilde:aA}. Likewise,
$$
\lim_{t\to+\infty}\frac{t\tilde{b}(t)}{\tilde{B}(t)}
=\lim_{s\to+\infty}\frac{sb(s)}{\tilde{B}(b(s))}=\lim_{s\to+\infty}\frac{sb(s)}{sb(s)-B(s)}
=\lim_{s\to+\infty}\frac{b(s)+sb'(s)}{sb'(s)}=\frac1{q}+1.
$$

Finally, Proposition \ref{proLH} (iv) and the fact that $\widetilde{(\widetilde A)}=A$, and $\widetilde{(\widetilde A)}=A$ ends the proof.

\end{proof}
\medskip

n order to prove some geometric properties of the functional $J$, we will made use of the following continuous embeddings :

\begin{lem}
\label{lem:cont:emb}
If $\al,\ \be\ge 0,$
\begin{equation}\label{emb1}
L^{\tilde A}(\Om)\hookrightarrow L^{\frac{p+1}{p}}(\Om)\ \hbox{and}\ L^{\tilde B}(\Om)\hookrightarrow L^{\frac{q+1}{q}}(\Om), 
\end{equation} 
and  if $\al<0,$ (respectively, $\be<0$), for all $\e>0$  small enough,
\begin{equation}\label{emb2}
L^{\tilde A}(\Om)\hookrightarrow L^{\frac{p+1}{p}-\e}(\Om), \qquad
\big(\text{respectively,}\quad L^{\tilde B}(\Om)\hookrightarrow L^{\frac{q+1}{q}-\e}(\Om)\big). 
\end{equation} 
Likewise, if $\al,\ \be\ge 0$, for all $\e>0$  small enough,
\begin{equation}\label{emb3}
L^{A}(\Om)\hookrightarrow L^{p+1 -\e}(\Om)\ \hbox{and}\ L^{B}(\Om)\hookrightarrow L^{q+1-\e}(\Om), 
\end{equation} 
and  if $\al<0,$ (respectively, $\be<0$),
\begin{equation}\label{emb4}
L^{A}(\Om)\hookrightarrow L^{p+1}(\Om), \qquad
\big(\text{respectively,}\quad L^{B}(\Om)\hookrightarrow L^{q+1}(\Om)\big). 
\end{equation} 
\end{lem}

\begin{proof}
We remark  that whenever $\al,\ \be\ge 0,$ \eqref{emb1} holds 
using that 
\begin{equation}\label{ineq}
a(t)\leq t^p \Longrightarrow \tilde{a}(t)\geq t^{1/p}\Longrightarrow
\widetilde{A}(t)\geq 
\frac{p}{p+1}t^{\frac{p+1}{p}}.   
\end{equation}
Moreover, as for all $u\in L^{\tilde A}(\Om)$, $u\not\equiv 0$,  we have from \eqref{2HIn2} in Proposition \ref{proLH}{\rm (iii)},
$$ 1\geq \int_\Om \tilde{A}\bigg( \frac{u}{\|u\|_{\tilde A}} \bigg) dx \geq \frac{p}{p+1} \int_\Om \bigg| \frac{u}{\|u\|_{\tilde A}}  \bigg|^{\frac{p+1}{p}}dx = \frac{p}{p+1} \frac{\|u \|_{\frac{p+1}{p}}^{\frac{p+1}{p}}}{\|u\|_{\tilde A}^{\frac{p+1}{p}}}\,,
$$
so, the embedding  \eqref{emb2}  is proved.

\bigskip

Besides, if $\al<0,$ (respectively, $\be<0$), for all $\e>0$  small enough, \eqref{emb2} holds.
Indeed, fix $\e >0$ small, there exists $c=c_\e>0$ such that 
$$
a(t)\leq\left\{ 
\begin{array}{lc}
ct^{p}& \hbox{ if } t\leq 1\\
ct^{p+\e} & \hbox{ if } t\geq  1,
\end{array}\right.
$$
then
$$
\tilde{a}(s)\geq 
\left\{
\begin{array}{lc}
c^{-1/p}s^{1/p}& \hbox{ if } s\leq c\\
c^{-1/(p+\e)}s^{1/(p+\e)} & \hbox{ if } s\geq  c,
\end{array}\right.
$$
and
$$
\tilde{A}(s)\geq \left\{
\begin{array}{lc}
\frac{p}{p+1} c^{-1/p}s^{(p+1)/p }& \hbox{ if } s\leq c\\
\frac{p+\e}{p+\e+1}c^{-1/(p+\e)}s^{(p+\e+1)/(p+\e)} + 
(\frac{p}{p+1} -\frac{p+\e}{p+\e +1})c & \hbox{ if } s\geq  c.
\end{array}\right.
$$
In particular 
$$
\tilde{A}(s)\geq C_1 s^{\frac{p+1}{p}-\e'} -C_2,
$$
for
$$
C_1:=\frac{p+\e}{p+\e+1}c^{-1/(p+\e)},
\quad C_2:=\max\left\{\frac{p+\e}{p+\e+1}c,\frac{\e\, c}{(p+\e +1)(p+1)}\right\},\quad  \e'=\frac{\e}{p(p+\e)}.
$$
with $C_1, C_2> 0$. From the above and by definition, if $u\in L^{\tilde{A}}(\Om)$, $u\not\equiv 0$, then
putting $v=\frac{u}{\|u\|_{\tilde A}}$
$$
1\geq\int_\Om \tilde A\big(v(x)\big)\, dx\ge \int_\Om C_1\, |v|^{\frac{p+1}{p}-\e'}dx -C_2|\Om|,
$$
so, 
$$
\big(1+C_2|\Om|\big)\,\|u\|_{\tilde A}^{\frac{p+1}{p}-\e'}\geq 
C_1\|u\|_{\frac{p+1}{p}-\e'}^{\frac{p+1}{p}-\e'}
$$
and the embedding \eqref{emb2} is proved.

The proofs of \eqref{emb3} and \eqref{emb4} are similar.
\end{proof}
\bigskip

The following Lemma is a technical one, that will be useful to prove   Proposition \ref{claim1}.  
\begin{lem}
\label{lem:tildeA:Atildea}
Let $a,\ A,\ \tilde{a},\ \tilde{A}$ be defined by \eqref{def:a:b}, \eqref{def:A:B}, \eqref{def:tilde:a:b}, and \eqref{def:tilde:A:B} respectively.
Here $\sim$ means {\it equivalent at infinity}, see  definition \ref{def:ll}. Then 
\begin{enumerate}
\item[\rm (i)] $A$ is equivalent to $\tilde{A}\circo a$ at infinity, $(A\sim  \tilde{A}\circo a)$, and $L^{A}(\Om)$ is isomorphic to $L^{ \tilde{A}\circo a}(\Om).$ 
\item[\rm {\rm (ii)}] $\tilde{A}$ is equivalent to $ A\circo\tilde{a}$ at infinity, $(\tilde{A}\sim  A\circo\tilde{a})$, and $L^{\tilde{A}}(\Om)$ is isomorphic to $L^{ A\circo\tilde{a}}(\Om)$. 
\end{enumerate}

\end{lem}
\begin{proof}
(i) Clearly, since Lemma \ref{lem:tilde:aA}, and Proposition \ref{pro:iso}, for any $\psi\in L^{ A}(\Om),$
$$
A(\psi)= A\big(|\psi|\big)\sim |\psi|^{p +1}\big(\ln |\psi|\big)^{-\alpha},\qq{and} \int_\Om  |\psi|^{p +1}\big(\ln |\psi|\big)^{-\alpha}\, dx\le C.
$$
Moreover,
$$
a(\psi)\sim |\psi|^{p}(\ln |\psi|)^{-\alpha},
$$
$$
\tilde{A}\big(a(\psi)\big)\sim \left(|\psi|^{p}(\ln |\psi|)^{-\alpha}\right)^{\frac{p +1}p}\big(\ln |\psi|\big)^{\frac{\alpha} p}=|\psi|^{p +1}\big(\ln |\psi|\big)^{-\alpha},
$$
so $A\sim  \tilde{A}\circo a$ at infinity, and
$$
\int_\Om \tilde{A}\big(a(\psi)\big)\, dx\le C, \qq{consequently}
a(\psi)\in L^{\tilde A}(\Om).
$$ 

{\rm (ii)} Likewise, for any $\psi\in L^{\tilde{A}}(\Om),$
$$
\tilde{A}(\psi)= \tilde{A}\big(|\psi|\big)\sim |\psi|^{\frac{p +1}p}\big(\ln |\psi|\big)^{\frac{\alpha} p},\qq{and} \int_\Om  |\psi|^{\frac{p +1}p}\big(\ln |\psi|\big)^{\frac{\alpha} p}\, dx\le C.
$$
Moreover,
$$
\tilde{a}(\psi)\sim 
|\psi|^\frac1{p}(\ln |\psi|)^{\frac{\alpha} p},
$$
and
$$
A\big(\tilde{a}(\psi)\big)\sim \left(|\psi|^\frac1{p}(\ln |\psi|)^{\frac{\alpha} p}\right)^{p+1}\big(\ln |\psi|\big)^{-\alpha}=
|\psi|^{\frac{p +1}p}\big(\ln |\psi|\big)^{\frac{\alpha} p},
$$
so $  \tilde{A}\sim A\circo \tilde{a}$ at infinity, and
$$
\int_\Om A\big(\tilde{a}(\psi)\big)\, dx\le C, \qq{consequently}
\tilde{a}(\psi)\in L^{A}(\Om).
$$ 
\end{proof}

\subsection{On the functional $J$}

Let us write $J'= \Phi-\Upsilon$, where $\Phi$  and $\Upsilon$ are defined, for all $(f,g)\in X$ and for all $(\psi_1,\psi_2)\in X$, as 
\begin{equation}\label{defPhi}
\Phi (f,g)[\psi_1,\psi_2]: =\int\limits_\Om \tilde{a}(f)\psi_1 \, dx +\int\limits_\Om\widetilde  b(g)\psi_2 \, dx,
\end{equation}
\begin{equation}\label{defUpsilon}\Upsilon(f,g)[\psi_1,\psi_2]:=\int\limits_\Om \psi_1\,\mathtt K (g)+\int\limits_\Om \psi_2\,\mathtt K (f) \, dx .
\end{equation}
\medskip

\noindent We will see that $\Upsilon$ is a continuous compact operator, see Proposition \ref{claim2}, and that $\Phi$  is an homeomorphism, see Proposition \ref{claim1} .\\

\medskip

We start by  proving that $\Upsilon$ is a compact operator. Let us recall that, by
using the regularity result Theorem \ref{regu}, $\big( \mathtt K (f),\mathtt K (g)\big) \in W^{2,\tilde A}(\Om)\times W^{2,\tilde B} (\Om)$ for all $(f,g)\in X$.  Furthermore, we have 

\begin{lem}\label{lem:comp:emb}
Let $p,q>0,\ \al,\be\in\R$ satisfying either \eqref{h:sub} or \eqref{h}.	
Then, the embeddings  $W^{2,{\tilde A}}(\Om)\hookrightarrow L^{B}(\Om)$ and   $W^{2,{\tilde B}}(\Om)\hookrightarrow L^{A}(\Om)$
are compact. 
\end{lem} 
\begin{proof}
We use   Theorem \ref{th:emb}. Assume first that $\frac{p+1}{p}\ge \frac{N}{2}.$
In particular, we  observe that if $\frac{p+1}{p}> \frac{N}{2}$, or  $\frac{p+1}{p}= \frac{N}{2}$ and $\al> \left(1-\frac{1}{N}\right)p \left(N-\frac{p+1}{p}\right)$, then $	W^{2,{\tilde A}}(\Om)\hookto C_b (\Om)$ is compact, where $C_b (\Om)=C(\Om)\cap L^\infty(\Om)$. Hence the embedding  
$$
W^{2,{\tilde A}}(\Om)\hookrightarrow L^{B}(\Om)
$$
is  also compact. 

Moreover, if $\frac{p+1}{p}= \frac{N}{2}$ and $\al\le  \left(1-\frac{1}{N}\right)p \left(N-\frac{p+1}{p}\right)$, then 
$$
W^{2,{\tilde A}}(\Om)\hookto L^{A_1} (\Om)
$$ 
is compact  for all ${\mathcal N}$-function $A_1$ such that 
$$
A_1\prec\hspace{-1.5mm}\prec\big({\widetilde A}^*\big)^*,
$$ 
where $\prec\hspace{-1.5mm}\prec$ is defined in \ref{def:dom}(iii),  and $\big({\widetilde A}^*\big)^*$,  at least  of exponential type, is defined by \eqref{def:Atilde:N:N}. Let us verify that $B\prec\hspace{-1.5mm}\prec\big({\widetilde A}^*\big)^*$ using that  $ B(s)\sim s^{q+1}(\ln s)^{-\be},$ see Lemma \ref{lem:tilde:aA} and Example \ref{example1}. 
Indeed, for all $c>0$  fixed, we have that 
\begin{equation*}
\lim\limits_{s\to+\infty}
\frac{{B}(s)}{\big({\widetilde A}^*\big)^*(c s)}=0, 
\end{equation*}
holds trivially.

Likewise, if $\frac{q+1}{q}\ge \frac{N}{2}$, then the embedding  $W^{2,{\tilde B}}(\Om)\hookrightarrow L^{A}(\Om)$, where $ A(s)\sim s^{p+1}(\ln s)^{-\al},$ is compact.\medskip

Let us now study the situation appearing when $\frac{p+1}{p}< \frac{N}{2}$.    
We  now remark (see Lemma \ref{lem:tilde:aA} and Example \ref{ex:HN:log})  that for $\frac{p+1}{p}< \frac{N}{2}$,
$$   
\widetilde A(s)\sim s^{\frac{p+1}{p}}(\ln s)^{\frac{\al}{p}},\ \hbox{ so then  }\ 
W^{2,{\tilde A}}(\Om)\hookto L^{A_1} (\Om)\qquad \forall A_1\prec\hspace{-1.5mm}\prec\big(({\widetilde A})^*\big)^*, 
$$
$$\qq{with} \big(({\widetilde A})^*\big)^*\sim
s^\frac{N(p+1)}{Np-2(p+1)}\,  	\big[\log(s)\big]^\frac{\al N}{Np-2(p+1)}.
$$

Let us verify that $B\prec\hspace{-1.5mm}\prec\big(({\widetilde A})^*\big)^*$, in other words, that for all $\delta>0$,
\begin{equation*}
\lim\limits_{s\to+\infty}
\frac{\big(({\widetilde A})^*\big)^*(\delta s)}{B(s)}=+\infty. 
\end{equation*}
Indeed, 
$$
\frac{\big(({\widetilde A})^*\big)^*(\delta s)}{B(s)}\sim  s^{\frac{N(p+1)}{Np-2(p+1)}-(q+1)}
\big[\log(s)\big]^{\frac{\al N}{Np-2(p+1)} +\be }
\to\infty ,
$$
if 
$$
\frac{N(p+1)}{Np-2(p+1)}> q+1, \qq{or} \frac{N(p+1)}{Np-2(p+1)}= q+1, \text{ and } \frac{\al N}{Np-2(p+1)} +\be>0,
$$
i.e. if 
$$
\frac{p}{p+1}-\frac{2}{N}<\frac{1}{q+1}, \qq{or}  \frac{N}{Np-2(p+1)}= \frac{q+1}{p+1}, \text{ and }  \frac{\al}{p+1}+ \frac{\be}{q+1}>0.
$$
This last inequalities are equivalent to
$$
\frac{N-2}{N}-\frac{1}{p+1}<\frac{1}{q+1},\qq{or}  \frac{1}{p+1}+\frac{1}{q+1}=\frac{N-2}{N}, \text{ and }  \frac{\al}{p+1}+ \frac{\be}{q+1}>0.
$$
The first inequality holds under hypothesis  \eqref{h:sub}, the second one holds by hypothesis  \eqref{h}.\\

Likewise if $\frac{q+1}{q}< \frac{N}{2}$, then the embedding $W^{2,{\tilde B}}(\Om)\hookrightarrow L^{A}(\Om)$
is compact, and the proof is achieved.
\end{proof}

\begin{pro}\label{claim2}
The linear operator $\Upsilon:X\to X'$ defined by
$$\Upsilon(f,g)[\psi_1, \psi_2]:= \int\limits_\Om \psi_1\,\mathtt K (g)\, dx +
\int\limits_\Om \psi_2\,\mathtt K (f) \, dx, \qquad \forall (\psi_1,\psi_2)\in X
$$
is a  continuous compact map.
\end{pro}
\begin{proof}
Let us prove the continuity.  Let $\{f_n\}_{n\in\N}$ be a sequence in $L^{\tilde A} (\Om)$ converging to some $f\in L^{\tilde A} (\Om)$   and    $\{g_n\}_{n\in\N}$ a  
sequence in $L^{\tilde B} (\Om)$ converging to some $g\in L^{\tilde B} (\Om)$. Let $\psi_1 \in L^{\tilde A} (\Om)$.
We have 
$$
\big| \Upsilon(f_n,g_n)[\psi_1, \psi_2]- 
\Upsilon (f,g)[\psi_1, \psi_2] \big| \leq   
\int_\Om \Big(\big|{\mathtt  K}(f_n)\psi_2-\mathtt K (f) \psi_2 \big|+
\big|{\mathtt  K}(g_n)\psi_1-\mathtt K (g) \psi_1\big|  \Big) \, dx
$$
$$ 
\leq C\Big(\big\|{\mathtt  K}(f_n)-\mathtt K (f)\big\|_{B}\, \|\psi_2\|_{\tilde B}
+ \big\|{\mathtt  K}(g_n)-\mathtt K (g)\big\|_{A}\, \|\psi_1\|_{\tilde A}\Big)
$$

We only have to prove that $\big\| \mathtt K (f_n) - \mathtt K (f)\big\|_B\to 0 $, the proof of $\big\| \mathtt K (g_n) - \mathtt K (g)\big\|_A\to 0 $ is analogous. This result follows from the  regularity results 
of Theorem \ref{regu} and the above Lemma \ref{lem:comp:emb}.
\medskip

Let us now  prove the compactness. Let $\{f_n\}_{n\in\N}$ be a bounded sequence in $L^{\tilde A} (\Om)$ and    $\{g_n\}_{n\in\N}$ a bounded 
sequence in $L^{\tilde B} (\Om)$. From \eqref{est2A} we have 
$$\|{\mathtt K}(f_n)\|_{W^{2, {\widetilde A}}}\leq C,\quad \|{\mathtt K}(g_n)\|_{W^{2, {\widetilde B}}}\leq C$$
for some constant $C$ independent of $n$. Then, using Lemma \ref{lem:comp:emb}  and the fact that $X$ is reflexive (cf. Proposition \ref{ref}, and Proposition \ref{proLH}) we infer the existence of $g\in L^{B} (\Om)$ and  $f\in L^{A} (\Om)$  such that, up to a subsequence, $\|{\mathtt K}(f_n) -g\|_{B}\to 0, \|{\mathtt K}(g_n) -f\|_{A}\to 0$ as $n\to\infty$.  Hence, using the 2nd Hölder's inequality,   for any $(\psi_1, \psi_2)\in X$, defining ${\mathcal I}_{f,g}[\psi_1,\psi_2] :=\int_\Om (f \psi_1 +g \psi_2)dx$,   $$
\big| \Upsilon(f_n,g_n)[\psi_1, \psi_2]- 
{\mathcal I}_{f,g}[\psi_1, \psi_2] \big| \leq   
\int_\Om \Big(\big|{\mathtt  K}(f_n)\psi_2-g \psi_2 \big|+
\big|{\mathtt  K}(g_n)\psi_1-f \psi_1\big|  \Big) \, dx
$$
$$ 
\leq C\Big(\big\|{\mathtt  K}(f_n)-g\big\|_{B}\, \|\psi_2\|_{\tilde B}
+ \big\|{\mathtt  K}(g_n)-f\big\|_{A}\, \|\psi_1\|_{\tilde A}\Big)
$$
so then $\| \Upsilon(f_n,g_n)- {\mathcal I}_{f,g} \|_{X'}\to 0 $ as  $n\to \infty$.
\end{proof}

\bigskip

We end this Section, proving that $\Phi$ is an homeomorphism.
\begin{pro}\label{claim1}
The non-linear operator $\Phi:X\to X'$ defined by
$$
\Phi (f,g)[\psi_1,\psi_2]:= \int\limits_\Om \tilde{a}(f)\psi_1 \, dx +\int\limits_\Om\tilde{b}(g)\psi_2 \, dx, \qq{for all} (\psi_1,\psi_2)\in X.
$$
is an homeomorphism.
\end{pro}

\begin{proof}

To prove that $\Phi:X\to X'$ is an homeomorphism, we use Proposition \ref{proLH}{\rm (iii)}-{\rm (iv)}, Remark \ref{tildeh:h-1},  and Lemma \ref{lem:tilde:aA} and Example \ref{example1}.\\

{\it Step 1. $\Phi$ is injective.}

\smallskip 

Assume that $\Phi (f_1,g_1)= \Phi (f_2,g_2)$, so for any $(\psi_1,\psi_2)\in X$
\begin{equation*}
0=\big(\Phi (f_1,g_1)- \Phi (f_2,g_2)\big) [\psi_1,\psi_2]=\int\limits_\Om \big[\tilde{a}(f_1)-\tilde{a}(f_2)\big]\psi_1 \, dx +\int\limits_\Om\big[\tilde{b}(g_1)-\tilde{b}(g_2)\big]\psi_2 \, dx,
\end{equation*}
then 
$$
\tilde{a}(f_1)=\tilde{a}(f_2),\qquad \tilde{b}(g_1)=\tilde{b}(g_2)\qq{a.e.}x\in\Om.
$$
or equivalently, since $\tilde{a}:=a^{-1}$ and $\tilde{b}:=b^{-1}$, 
$$
f_1=f_2,\qquad g_1=g_2\qq{a.e.}x\in\Om.
$$

{\it Step 2.  $\Phi$ is surjective.} 

\smallskip 

Let us check that for any $(\psi_1,\psi_2)\in X'=L^{ A}(\Om)\times L^{ B}(\Om)$ (see 
\eqref{dual:X}), $\big(a(\psi_1), b(\psi_2)\big)\in X,$ and
\begin{equation}\label{Phi:surj}
\Phi\big(a(\psi_1), b(\psi_2)\big)=(\psi_1,\psi_2).
\end{equation}
Since Lemma \ref{lem:tildeA:Atildea}, $a(\psi_1)\in L^{\tilde A}(\Om)$. Likewise $b(\psi_2)\in L^{\tilde B}(\Om),$ and $\big(a(\psi_1), b(\psi_2)\big)\in X.$

\smallskip

Since $\tilde{a}:=a^{-1}$ and $\tilde{b}:=b^{-1}$, it is straightforward to check that \eqref{Phi:surj} holds.\\

Consequently,
$$
\Phi^{-1}(\psi_1,\psi_2)= \big(a(\psi_1), b(\psi_2)\big). 
$$

{\it Step 3. $\Phi$ is continuous.}

\smallskip 

Let $(f_n,g_n)\to (f,g)$ be a  convergent sequence in $L^{\tilde A} (\Om)\times L^{\tilde B} (\Om)$, in other words
\begin{equation*}
\|f_n-f\|_{\tilde A}+\|g_n-g\|_{\tilde B}\to 0.
\end{equation*} 
\noindent {\it Claim.}
\begin{equation*}
\|f_n-f\|_{\tilde A}\to 0 \implies
\big\|\tilde{a}(f_n)-\tilde{a}(f)\big\|_{A}\to 0. 
\end{equation*}
Once proved the claim 
$$
\big\|\tilde{a}(f_n)-\tilde{a}(f)\big\|_{A}
+  \big\|\tilde{b}(g_n)-\tilde{b}(g)\big\|_{B}\to 0. 
$$
Then, since Proposition \ref{proLH}(i)-{\rm (ii)}, and Remark \ref{remLH}, for all $(\psi_1,\psi_2)\in X,$  
\begin{align*}
\Big|\big(\Phi (f_n,g_n)- \Phi (f,g)\big) [\psi_1,\psi_2]\Big|&=\left|\int\limits_\Om \big[\tilde{a}(f_n)-\tilde{a}(f)\big]\psi_1 \, dx +\int\limits_\Om\big[\tilde{b}(g_n)-\tilde{b}(g)\big]\psi_2 \, dx\right| \\ 
&\le  
\big\|\tilde{a}(f_n)-\tilde{a}(f)\big\|_{A} \|\psi_1\|_{\tilde{A}}
+ \big\|\tilde{b}(g_n)-\tilde{b}(g)\big\|_{B} \|\psi_2\|_{\tilde{B}}
\to 0,
\end{align*}
and the continuity is achieved.\\

\noindent		{\it Proof of the Claim.} 

Since Theorem \ref{th:conv:mean=norm}, to conclude the claim, we need to check  that $\tilde{a}(f_n)\to \tilde{a}(f)$ in $A$-mean, in other words
$$
\int_\Om A\big(\big|\tilde{a}(f_n)-\tilde{a}(f)\big|\big)\, dx\to 0.
$$

On the one hand since Lemma \ref{lem:tildeA:Atildea},
$ \tilde{A}\sim A\circo \tilde{a}$ at infinity, and $L^{\tilde{A}}(\Om)$ is isomorphic to $L^{ A\circo\tilde{a}}(\Om)$,
so, for any $f\in L^{\tilde A} (\Om)$
$$
\int_\Om A\big(\tilde{a}(f)\big)\, dx\le C,\qq{and consequently}
\tilde{a}(f)\in L^{A}(\Om).
$$
Likewise, $\{\tilde{a}(f_n)\}\subset L^{A}(\Om)$.  

\smallskip

On the other hand
since Theorem \ref{th:conv:mean=norm}, $f_n\to f$ in ${\tilde A}$-mean, so
\begin{equation}
\label{e:c1}
\int_\Om {\tilde A}\big(f_n-f\big)\, dx=\int_\Om {\tilde A}\big(|f_n-f|\big)\, dx\to 0.
\end{equation}

Assume by contradiction that the claim do not hold.
Then, there exist $\varepsilon_0>0$  and a subsequence $\{f_{n_k}\}_k$ satisfying
\begin{equation}
\label{e:cc}
\int_\Om A\left(\tilde{a}(f_{n_k})-\tilde{a}(f)\right)dx>\varepsilon_0  \quad \mbox{ for all } k\in \N. \end{equation}

Put $z_k:=\widetilde{A}(|f_{n_k}-f|)$. Since \eqref{e:c1}, $z_k\to 0$ in $L^1(\Om)$, so  by the Lebesgue reverse dominated converge theorem, there exists a subsequence $\{z_{k_j}\}_j$ and there exists $h\in L^1(\Om)$ such that 
$$z_{k_j}\leq h ,\quad z_{k_j}\to 0 \hbox{ a.e.},\hbox{ and also }  \tilde{A}^{-1}z_{k_j}=f_{n_{k_j}}-f\to 0 \hbox{ a.e.} $$
We have 
$$
|f_{n_{k_j}}|\leq \widetilde{A}^{-1}(h)+|f|\in L^{\widetilde{A}}(\Om).  
$$
By monotonicity of 
$\widetilde{a}$, and since Lemma \ref{lem:tildeA:Atildea},
$ \tilde{A}\sim A\circo \tilde{a}$ at infinity, so
$$
\widetilde{a} (f_{n_{k_j}})\leq  \widetilde{a}\left( \widetilde{A}^{-1}(h)+|f|\right):=\ell\in L^A(\Om),
$$
moreover
$$
A\left(|\widetilde{a} (f_{n_{k_j}})-\widetilde{a} (f)|\right)\leq A(\ell+\widetilde{a}(f))\in L^1(\Om).
$$
Since pointwise convergence
$A\left(|\widetilde{a} (f_{n_{k_j}})-\widetilde{a} (f)|\right) \to 0$ a.e., and from Lebesgue
dominate convergence theorem we conclude that
$$\int_\Om A\left(|\widetilde{a} (f_{n_{k_j}})-\widetilde{a} (f)|\right)dx \to 0,$$
which contradits \eqref{e:cc}. This concludes the proof of the claim.
\end{proof}

\section{Proof of Theorem 1.1}
\label{sec:proof}
First we star by proving that $J$ satisfies the $(PS)$ condition.
To this aim we will use  of the following (trivial) lemma: 
\begin{lem}[\cite{Dos-S}, Lemma 3.1] Let $X$ be a reflexive Banach space and $F\in C^1(X)$ be such that
\begin{enumerate}
\item[\rm (1)] any Palais-Smale sequence of $F$ is bounded;
\item[\rm (2)]  for all $u\in X$,
$$F'(u)= L(u)-T(u),$$
where $L: X\to  X'  $ is a homeomorphism and $T: X\to X'$ is a continuous compact map.
\end{enumerate} Then, $F$ satisfies the Palais-Smale condition.
\end{lem}
\begin{pro}
$J$ satisfies the $(PS)$ condition.
\end{pro}

\begin{proof}
Let us first check  (2)   for the space $X=L^{\widetilde{A}}(\Om)\times L^{\widetilde{B}}(\Om)$ and the functional $F=J$ defined \eqref{phi}.  Observe that
$X$
is a reflexive Banach space by Proposition \ref{ref} and  $J\in C^1(X;\R)$
is such that
$J'=\Phi-\Upsilon $
where
$\Phi$  is defined in \eqref{defPhi}
and $\Upsilon$ is defined in \eqref{defUpsilon}.
Thus, (2) follows from  Proposition \ref{ref} and Proposition  \ref{claim2}.
\medskip

Now we check (1). 
Let $\{(f_n,g_n)\}_{n\in\N}$ be a Palais-Smale sequence for the functional $J$, that is, we have for some $C>0$ and some sequence of positive  real numbers
$\{\varepsilon_n\}_{n\in\N}\to 0$,\\

(PS1):$\quad \quad\  J(f_n,g_n)\leq C \hbox{ for all  } n\in\N,$\\

(PS2): $\quad \quad  \left|J' (f_n, g_n)[\psi_1, \psi_2]\right|\leq \varepsilon_n \|(\psi_1, \psi_2)\|_X \qquad \forall (\psi_1, \psi_2)\in X. $

\bigskip

Since  \eqref{K:na}, for every $(f,g)\in X$,  $\te\in \R$ we have
$$
\frac12\Upsilon(f,g)[f,g]= \frac12\left[\int\limits_\Om f\,\mathtt K (g)+\int\limits_\Om g\,\mathtt K (f) \, dx \right]=
\Upsilon(f,g)\left[\te f,(1-\te)g\right],
$$
and since the identity
$$
J (f_n,g_n)-J' (f_n,g_n) \left[\te f,(1-\te)g\right]=
\int_\Om \big(\tilde{A}(f_n)+\tilde{B}(g_n)-\te f_n \tilde a(f_n)-(1-\te)g_n \tilde b(g_n)\big)\,dx
$$
it follows from $(PS1)$-$(PS2)$, 
\begin{equation*}
\int_\Om \left(\tilde{A}(f_n)-\te f_n \tilde a(f_n)\right)dx +\int_\Om \left(\tilde{B}(g_n)-(1-\te)g_n \tilde b(g_n)\right)\,dx \leq \varepsilon_n \|(f_n,g_n)\|_X +2C.
\end{equation*}
Assume by contradiction that 
$\|(f_n,g_n)\|_X\rightarrow +\infty$.
Using the Young equality, see Proposition \ref{pro:H*} {\rm (iii)},
\begin{equation}\label{ye}
\widetilde{A}(t)= t\tilde{a}(t)-A\big(\tilde{a}(t)\big) 
\end{equation}
Integrating by parts $A(t)=\int_0^t \frac{s^p}{(\ln (e+s))^\alpha}\, ds,$  we have that
\begin{equation*}
\displaystyle A(t)=\frac{t\,a(t)}{p+1}+\frac{\alpha}{p+1}\int^t_0 \frac{s^{p+1}}{\big(\ln(e+s)\big)^{\alpha+1}}\,\frac{ds}{e+s},
\end{equation*}
and consequently,
\begin{equation*}
\displaystyle A\big(\tilde{a}(t)\big)=\frac{1}{p+1}\, t\,\tilde{a}(t)+\frac{\alpha}{p+1}\int^{\tilde{a}(t)}_0 \frac{s^{p+1}}{\big(\ln(e+s)\big)^{\alpha+1}}\,\frac{ds}{e+s}.
\end{equation*}
Likewise 
\begin{equation*}
B(t)=\frac{1}{q+1}\, t\,b(t)+\frac{\beta}{q+1}\int^t_0 \frac{s^{q+1}}{\big(\ln(e+s)\big)^{\beta+1}}\,\frac{ds}{e+s},
\end{equation*}
and 
\begin{equation*}
\displaystyle B\big(\tilde{b}(t)\big)=\frac{1}{q+1}\, t\,\tilde{b}(t)+\frac{\be}{q+1}\int^{\tilde{b}(t)}_0 \frac{s^{q+1}}{\big(\ln(e+s)\big)^{\be+1}}\,\frac{ds}{e+s}.
\end{equation*}
Hence
\begin{align*}
\widetilde{A}(t)-\te t\tilde{a}(t)
&=\left(\frac{p}{p+1}-\te\right)\, t\,\tilde{a}(t)-\frac{\alpha}{p+1}\int^{\tilde{a}(t)}_0 \frac{s^{p+1}}{\big(\ln(e+s)\big)^{\alpha+1}}\,\frac{ds}{e+s},
\end{align*}
and
\begin{align*}
\widetilde{B}(t)-(1-\te)t\,\tilde{b}(t) 
&=\left(\frac{q}{q+1}-(1-\te)\right)\, t\,\tilde{b}(t)-\frac{\be}{q+1}\int^{\tilde{b}(t)}_0 \frac{s^{q+1}}{\big(\ln(e+s)\big)^{\be+1}}\,\frac{ds}{e+s}.
\end{align*}
Since l'Hôpital rule,
\begin{equation*}
\lim_{t\to\infty}\frac{\D\int^t_0 \frac{s^{p+1}}{\big(\ln(e+s)\big)^{\alpha+1}}\,\frac{ds}{e+s}}{\D\frac{t^{p+1}}{\big(\ln(e+t)\big)^{\alpha+1}}}=\frac{1}{p+1},
\qq{moreover} \lim_{t\to\infty}\frac{\D\frac{t^{p+1}}{(\ln(e+t))^{\alpha+1}}}{ta(t)}=0,     
\end{equation*}
then, for each $\e>0$ there exists a $C_\e$ such that
\begin{equation*}
\frac{\alpha}{p+1}\int^t_0 \frac{s^{p+1}}{\big(\ln(e+s)\big)^{\alpha+1}}\,\frac{ds}{e+s}
\le \e\,ta(t)+C_\e,\qq{for all} t\ge 0,
\end{equation*}
and
\begin{equation*}
\int^{\tilde{a}(t)}_0 \frac{s^{p+1}}{\big(\ln(e+s)\big)^{\alpha+1}}\,\frac{ds}{e+s}
\le \e\,t\tilde{a}(t)+C_\e,\qq{for all} t\ge 0.
\end{equation*}
Choosing
$$
\te=\frac12\left(1-\frac{1}{p+1}+\frac{1}{q+1}\right)=\frac12\left(\frac{p}{p+1}+\frac{1}{q+1}\right)\in[0,1],
$$
we get that
$$
\frac{p}{p+1}-\te=\frac{q}{q+1}-(1-\te)=\frac12\left(1-\frac{1}{p+1}-\frac{1}{q+1}\right)>0.
$$
Consequently
\begin{align*}
\widetilde{A}(t)-\te t\tilde{a}(t)
&\ge \left[\frac12\left(1-\frac{1}{p+1}-\frac{1}{q+1}\right)-\e\right]\, t\,\tilde{a}(t)-C_e,
\end{align*}
and
\begin{align*}
\widetilde{B}(t)-(1-\te)t\,\tilde{b}(t) 
&\ge \left[\frac12\left(1-\frac{1}{p+1}-\frac{1}{q+1}\right)-\e\right]\, t\,\tilde{b}(t)-C_e.
\end{align*}
Hence
\begin{align*}\label{AB}
&\int_\Om  f_n\,\tilde{a}(f_n) \,dx  + \int_\Om g_n\,\tilde{b}(g_n)
\,dx 
\leq \varepsilon_n \|(f_n,g_n)\|_X +C,
\end{align*}
and as $t\tilde{a}(t)>0$ for all $t$, we then get 
\begin{equation*}\label{AB:3}
\frac{1}{\|(f_n,g_n)\|_X}
\int_\Om  f_n\,\tilde{a}(f_n) \,dx\to 0 
,\qquad \frac{1}{\|(f_n,g_n)\|_X}\, \int_\Om g_n\,\tilde{b}(g_n)
\,dx \to 0.
\end{equation*}
Thus, using that for all $u\in L^{H}(\Om)$ 
$$
\|u\|_{H}\leq 2\|u\|_{(H)}\leq 2\max\left\{ \int_\Om H (u)\,dx, 1\right\}
$$
(cf. Lemma \ref{Ban}{\rm (iii)} and Proposition \ref{proLH}(i)) for  $H=\tilde{A}$ and  $H=\tilde{B}$,
we reach a contradiction.
\end{proof}

\noindent{\it Proof of Theorem \ref{main}}. We wish to apply a mountain pass type  theorem. 
We divide the proof in 2 steps.
Recall that the functional $J$ is defined in \eqref{phi}.\medskip

\noindent {\it Step 1}. We prove that  $(0,0)$ is a local minimum of $J $.   By Young's inequality for both $A$ and $B$ implies 
\begin{align*}J (f,g)=&\int\limits_\Om \widetilde A(f)\, dx +\int\limits_\Om\widetilde B(g)\, dx -\frac{1}{2}\int\limits_\Om \left(f\,\mathtt K (g)+g\,\mathtt K (f)\right)\, dx\\
\geq& \int\limits_\Om \widetilde A(f)\, dx +\int\limits_\Om\widetilde B(g)\, dx-\frac{1}{2} \int_\Om \left(\widetilde{A}(f) + A(\mathtt K (g))\right) \, dx 
-\frac{1}{2}\int_\Om \left(\widetilde{B}(g) + B(\mathtt K (f))\right) \, dx \\
=& \frac{1}{2}\int\limits_\Om \widetilde A(f)\, dx +\frac{1}{2}\int\limits_\Om\widetilde B(g)\, dx
-\frac{1}{2}\int\limits_\Om A (\mathtt K (g))\, dx- \frac{1}{2} \int_\Om  B(\mathtt  K (f)) \, dx.
\end{align*}
Let us distinguish 3 cases according to the signs of $\alpha$ and $\beta$.
\medskip

\noindent {\it Case 1: Assume that  $\alpha\geq 0$ and $\beta\geq 0$.}
First observe  that, since  
$$\widetilde A (t)\geq \frac{p}{p+1}t^{\frac{p+1}{p}}, \quad  \widetilde B (t)\geq \frac{q}{q+1}t^{\frac{q+1}{q}}  
$$
for all $t\geq 0$ (see \eqref{ineq}), then $X\subset L^{\frac{p+1}{p}}(\Om)\times L^{\frac{q+1}{q}}(\Om) $. Moreover,
using  that 
$$ 
A(t)\leq \frac{1}{p+1}t^{p+1},\quad B(t)\leq \frac{1}{q+1}t^{q+1}
$$ 
for all $t\geq 0$, we get
$$ 
\int_\Om A(\mathtt K(g))dx\leq \frac{1}{p+1}\|\mathtt K (g)\|_{p+1}^{p+1}, \quad  \int_\Om B(\mathtt K(f))dx\leq \frac{1}{q+1}\|\mathtt K (f)\|_{q+1}^{q+1} .
$$
Besides, using Sobolev embeddings, either the second inequality of \eqref{h:sub} or the equality in \eqref{h} and the regularity of $\mathtt K$, 
there exists  a constant $C>0$ independent of $f$ and $g$ such that 
$$
\|\mathtt K (g)\|_{p+1}\leq C\| \mathtt K( g)\|_{W^{2, \frac{q+1}{q}}}\leq C\|g\|_{\frac{q+1}{q}}, \quad \|\mathtt K (f)\|_{q+1}\leq C\| \mathtt K (f)\|_{W^{2, \frac{p+1}{p}}}\leq C\|f\|_{\frac{p+1}{p}}.
$$
Thus
$$ 2J(f,g)\geq \bigg(\frac{p}{p+1} \|f\|_{\frac{p+1}{p}}^{\frac{p+1}{p}} - C^{q+1}\|f\|_{\frac{p+1}{p}}^{q+1}\bigg)+ \bigg(\frac{q}{q+1} \|g\|_{\frac{q+1}{q}}^{\frac{q+1}{q}} - C^{p+1}\|g\|_{\frac{q+1}{q}}^{p+1}\bigg).  $$
Finally, using the first inequality of \eqref{h:sub}  and \eqref{h} (i.e. $\frac{p+1}{p}<q+1$, or equivalently, $\frac{q+1}{q}<p+1$) and  that
$\|f\|_{\frac{p+1}{p}}\leq \left(\frac{p+1}{p}\right)^\frac{p}{p+1}\|f\|_{\widetilde A},\quad \|g\|_{\frac{q+1}{q}}\leq \left(\frac{q+1}{q}\right)^\frac{q}{q+1}\|g\|_{\widetilde B}$
,  we conclude that 
$J(f,g)> 0
$
for all $(f,g)\in X$ with $0\not=\|(f,g)\|_X$ sufficiently small.

\bigskip

\noindent {\it Case 2: Assume that  $\alpha\ge 0$ and $\beta<0$ (the case $\alpha< 0$   and $\beta\ge 0$ is analogous).}

Using Sobolev embeddings,  the  regularity of $\mathtt K$, 
and since $q+1\le \left(\frac{N-2}{N}-\frac1{p +1}\right)^{-1} =\left(\frac p{p +1}-\frac{2}{N}\right)^{-1},$
we have   that 
$$
\|\mathtt K (f)\|_{q+1}\leq C \|\mathtt K (f)\|_{\left(\frac p{p +1}-\frac{2}{N}\right)^{-1}}\leq C\| \mathtt K (f)\|_{W^{2,\frac{p +1}p}}\leq C\|f\|_{\frac{p +1}p}\,\,.
$$

On the one hand, since \eqref{eq:A:p+1:>},
and the above
$$
\int_\Om  B (\mathtt K (f))\, dx\le \frac{1}{q+1}\,  \|\mathtt K (f)\|_{q+1}^{q +1}\le 
C\|f\|_{\frac{p +1}p}^{q +1} \,\,.
$$

On the other hand, since \eqref{eq:tilde:A:p+1}
$$
\int_\Om  {\widetilde A} (f)\, dx\ge C\,  \|f\|_{\frac{p +1}p}^{\,\frac{p +1}p}\,\,.
$$

Hence, using the Young inequality
$$
\frac{1}{2}\int\limits_\Om \left(f\,\mathtt K (g)+g\,\mathtt K (f)\right)\, dx
=\int\limits_\Om f\,\mathtt K (g)\, dx
\le \int_\Om \tilde B (g) dx + \int_\Om B(K(f)) dx 
$$ 
so 
$$
J(f,g)\geq \int_\Om \tilde A (f) dx -\int_\Om B(K(f))dx
$$ and using (2.41) 

$$J(f,g)\geq C\|f\|_{\frac{p +1}p}^{\frac{p+1}{p}} -D\| f\|_{\frac{p +1}p}^{q+1}> 0,
$$
for $0\not=\|f\|_{\tilde A}$ small enough. Observe that if $f=0$, then $J(f,g)=\int_\Om \tilde B (g) dx>0$ for $\|g\|_{\tilde B}\ne 0.$

\bigskip

\noindent {\it Case 3: Assume that  $\alpha< 0$   and $\beta< 0$.  }

Notice that if $\alpha<0,\, \beta <0$ then necessarilly \eqref{h:sub} holds.
Using Orlicz-Sobolev embeddings (see Lemma \ref{lem:comp:emb}), 
and the regularity of $\mathtt K$ (see Theorem \ref{regu}), we have   that 
$$
\|\mathtt K (g)\|_{A}\leq C\| \mathtt K( g)\|_{W^{2, \tilde B}}\leq C\|g\|_{\tilde B}, \quad 
\|\mathtt K (f)\|_{B}\leq C\| \mathtt K (f)\|_{W^{2,\tilde  A}}\leq C\|f\|_{\tilde  A}.
$$

On the one hand since \eqref{eq:A:<2}, 
and the above
$$
\int_\Om  A (\mathtt K (g))\, dx\le C\,  \|\mathtt K (g)\|_{A}^{p +1}\le 
C\|g\|_{\tilde B}^{p +1} ,\quad \int_\Om B(\mathtt K(f))dx\le C\|\mathtt K (f)\|_{B}^{q+1}\le C\|f\|_{\tilde  A}^{q+1}.
$$
whenever $\|\mathtt K (g)\|_{A}\le 1,\ \|\mathtt K (f)\|_{B}\le 1.$

On the other hand since \eqref{eq:tilde:A:>2} in Lemma \ref{lem:norm:tilde:A},  and Lemma \ref{lem:cont:emb},
$$
\int_\Om  {\widetilde A} (f)\, dx\ge C\,  \|f\|_{(\tilde A)}^{\,\frac{p +1}p},
\quad 
\int_\Om  {\widetilde B} (g)\, dx\ge C\,  \|g\|_{\tilde B}^{\frac{q +1}q},
$$
whenever $\|f\|_{(\tilde A)}\le 1,\ \|g\|_{\tilde B}\le 1.$

Then, for all $(f,g)\in X$ with $0\not=\|(f,g)\|_X$ sufficiently small
$$ 
2J(f,g)\geq d_1\|f\|_{(\tilde A)}^{\,\frac{p +1}p} - d_2\|f\|_{\tilde A}^{q+1}+ \bigg(d_3\|g\|_{\tilde B}^{\frac{q+1}{q}} - d_4\|g\|_{\tilde B}^{p+1}\bigg)>0. 
$$

\medskip
\noindent {\it Step 2}. There exists $(f_1,g_1)\in X$ such that  $J(f_1,g_1)<0$.
Indeed, choose $s\in  (1/p,q)$ (this is possible as $qp>1$, which follows from \eqref{h:sub} or \eqref{h}).
We shall show that there exists $t\in \R$ big enough such that  $J(t\phi_1,t^s\phi_1)<0$. Indeed,  since $\mathtt K (\phi_1)=\phi_1/\la_1,$ we can write
\begin{align*}
J (t\phi_1,t^s\phi_1)=&\int\limits_\Om \widetilde A(t\phi_1)\, dx  +\int\limits_\Om\widetilde B(t^s\phi_1)\, dx 
-\frac{t^{1+s}}{\la_1}\int\limits_\Om (\phi_1)^2\, dx.
\end{align*}
Since Lemma \ref{lem:tilde:aA} and Example \ref{example1}, $\widetilde A(t\phi_1)\sim {(t\phi_1)^{\frac{p +1}p}(\ln (t\phi_1))^\frac\alpha p}$, 
and $\widetilde B(t^s\phi_1)\sim {(t^s\phi_1)^{\frac{q+1}{q}}(\ln (t^s\phi_1))^\frac{\beta}{q}}$.
Since  $\frac{p+1}{p}<1+s$, and also  $s\frac{q+1}{q}<1+s$,
hence $J (t\phi_1,t^s\phi_1)\to -\infty$ as $t\to \infty.$

\section{A non-existence result}
\label{sec:crit}
The following lemma provides  a Rellich-Pohozaev-Mitidieri type identity, see \cite{Pohozaev, Pucci_Serrin, Mitidieri}. 

\begin{lem}\label{PRMidentity}(Rellich-Pohozaev-Mitidieri type identity)\\
Let $u$ and $v$ be in $C^2(\bar{\Omega})$, where $\Omega$ is a $C^1$ domain in $\mathbb{R}^N$,  and $u=v=0$ on $\partial{\Omega} $. Then

\begin{align*}
\displaystyle \int_{\Omega} \Delta u\; (x\cdot \nabla v)+ \Delta v\; (x\cdot \nabla u)&= (N-2)\int_{\Omega}( \nabla u\cdot  \nabla v) +\int_{\partial\Omega}\frac{\partial u}{\p \nu} \;(x\cdot \nabla v)\nonumber\\ 
& +  \int_{\partial\Omega}\frac{\partial v}{\p \nu}\;(x\cdot \nabla u) -\int_{\partial\Omega}(\nabla u,\nabla v)\; (x\cdot   \nu),
\end{align*}
where $\nu$ denotes the exterior normal, and $(x\cdot \nu)$ denotes the inner product.
\end{lem}
For the proof we refer to  \cite{Mitidieri}.

By definition, a solution $(u,v)$ of \eqref{s} belongs to $\big(W_0^{1, {\widetilde A}}(\Om)\cap W^{2,{\tilde A}}(\Om)\big)\times\big( W_0^{1, {\widetilde B}}(\Om)\cap W^{2,{\tilde B}}(\Om)\big)$. Since an estimate of Brezis-Kato \cite{Brezis_Kato}, based on Moser iteration techniques \cite{Moser}, $u,v\in C^{1, \nu}(\Omb)\cap W^{2,s}(\Om)$ for any $\nu<1,\ s<+\infty$, see for instance \cite[Lemma 2.1 (iv)]{Pardo_JFPTA}.
If, in particular $\Om$ is $C^{2,\mu}$, then $u,v\in C^{2, \mu}(\Omb).$
\bigskip

\begin{thm}
[\bf Non existence of non-negative non-trivial classical solutions]\label{th:non:ex}
Let $(u,v) \in \big(C^2(\Omb)\big)^2$ be a pair of non-negative  solutions to the problem \eqref{s}. 

Assume that 
\begin{equation}\label{h:crit}
\frac1{p+1}+\frac1{q+1}=\frac{N-2}N.
\end{equation}
Assume also that $\Om$ is strongly star-shaped  with respect to $0$ (the inner product $(x\cdot \nu(x))\ge 0$ for all $x\in \p\Om$, and $(x\cdot \nu(x))\not\equiv 0$) and that $\p\Om$ is $C^1$. 
If $\al,\be\le 0$, then
$$
u\equiv 0,\quad  v\equiv 0,\qq{in}\Om. 
$$

Assume now that 
\begin{equation}\label{h:crit:2}
\frac1{p+1}+\frac1{q+1}<\frac{N-2}N.
\end{equation}
and that $\Om$ is star-shaped  with respect to $0$ ($(x\cdot \nu(x))\ge 0$ for all $x\in \p\Om$) and that $\p\Om$ is $C^1$. 
If $\al,\be\le 0$, then
$$
u\equiv 0,\quad  v\equiv 0,\qq{in}\Om. 
$$

\end{thm}

\begin{rem}\label{rem}
In particular, for the single equation, i.e.  $p=q$ and $\al=\be\le 0$,   we have the following:

Assume that $p=2^*-1$, and also that $\Om$ is strongly star-shaped  with respect to $0$.
If $\al\le 0$, then
$$
u\equiv 0,\qq{in}\Om. 
$$

Assume now that 
$p>2^*-1$,
and that $\Om$ is star-shaped  with respect to $0$. 
If $\al\le 0$, then
$$
u\equiv 0,\qq{in}\Om. 
$$
\end{rem}
\begin{proof}[Proof of Theorem \ref{th:non:ex}]

If we set $W(u,v):= A(u)+B(v)$ then $W_u=a(u)$ and  $W_v=b(v)$.
Therefore, for solutions $u>0$ and $v>0$  of  \eqref{s}, 
\begin{align*}
&\displaystyle  \int_{\Om} - \left[\Delta u\; (x\cdot \nabla v)\right. +\left.\Delta v\; (x\cdot \nabla u)\right]  
= \int_{\Om} \sum_j x_j\left(\frac{\p W}{\p v}\frac{\p v}{\p x_j}+\frac{\p W}{\p u}\frac{\p u}{\p x_j}\right)=\int_{\Om} \sum_j x_j \frac{\p W}{\p x_j}\\
&\qquad=-N \int_{\Om} W+\int_{\Om} {\rm{div}}(W\,x )
= -N \int_{\Om} \left[A(u)+B(v)\right]+ \int_{\p\Om}(x\cdot \nu)\;W(u,v),
\end{align*}
and 
\begin{equation}\label{graduv}
\displaystyle \int_{\Om} \nabla u \nabla v= \int_{\Om}u a(u)
= \int_{\Om}v b(v).    
\end{equation}

Applying Lemma \ref{PRMidentity} (Pohozaev-Rellich-Mitidieri type identity) we get that 
\begin{align}\label{PRM:1}
&\displaystyle N \int_{\Om} \left[A(u)+B(v)\right] -{(N-2)} \int_{\Om}\,u\,a(u)
\\& =  \int_{\p\Om}(x\cdot \nu)\;W(u,v) -\int_{\p\Om}(\nabla u\cdot \nabla v)\; (x\cdot   \nu) + \int_{\p\Om}\frac{\p u}{\p \nu} \;(x\cdot \nabla v) +  \int_{\p\Om}\frac{\p v}{\p \nu}\;(x\cdot \nabla u).\nonumber
\end{align}

On the one hand, $u=v=0$ on the boundary, so the first integral in the r.h.s. vanishes. Moreover, since $u=0$ on $\p\Om$, the tangential component of $\na u$ vanishes and $\na u$ is parallel to the normal vector $\nu(x)$ at each point $x\in\Om$, in other words $\na u(x)=\pm |\na u(x)|\,\nu(x)$.

On the other hand, since Hopf's Lemma, $\frac{\partial u}{\p \nu}(x)<0$ for all $x\in\p\Om$,  so  $\na u(x)\cdot\nu=\frac{\partial u}{\p \nu}(x)=- |\na u(x)|$, and consequently, $\na u(x)=- |\na u(x)|\,\nu(x).$  Likewise,
$\frac{\partial v}{\p \nu}(x)=- |\na v(x)|$, and $\na v(x)=- |\na v(x)|\,\nu(x).$ 
Consequently, the  r.h.s. of \eqref{PRM:1} is reduced to
\begin{align*} 
&\Big(-|\na u(x)||\na v(x)| + |\na u(x)||\na v(x)| +|\na u(x)||\na v(x)|\Big)\, (x\cdot \nu)\\
&\qquad =|\na u(x)||\na v(x)|\, (x\cdot \nu)=\frac{\partial u}{\p \nu}\frac{\partial v}{\p \nu}\, (x\cdot \nu),
\end{align*}	
hence
\begin{align}\label{PRM:2}
&\displaystyle N \int_{\Om} \left[A(u)+B(v)\right] -{(N-2)} \int_{\Om}\,u\,a(u)
=  \int_{\p\Om}\frac{\partial u}{\p \nu}\frac{\partial v}{\p \nu}\, (x\cdot \nu).
\end{align}
Integrating by parts $A(t)=\int_0^t \frac{s^p}{(\ln (e+s))^\alpha}\, ds,$  
we have that
		\begin{equation}\label{A}
			\displaystyle A(t)-\frac{1}{p+1} ta(t)
=\frac{\alpha}{p+1}\int^t_0 \frac{s^{p+1}}{\big(\ln(e+s)\big)^{\alpha+1}}\,\frac{ds}{e+s}=\frac{\alpha}{p+1}\int^t_0 \frac{a(s)}{\ln(e+s)}\frac{s}{e+s}\,ds.
		\end{equation}
		Likewise 
		\begin{equation}\label{B} 
\displaystyle B(t)-\frac{1}{q+1}\, t\,b(t)
=\frac{\beta}{q+1}\int^t_0 \frac{b(s)}{\ln(e+s)}\frac{s}{e+s}\,ds.
		\end{equation}

Substituting \eqref{A}-\eqref{B} into \eqref{PRM:2}, and since \eqref{graduv},
\begin{align}\label{PRM:4a}
&\left[N\left( \frac{1}{p+1}+\frac{1}{q+1}\right)-(N-2)\right]\int_{\Om}\,u\,a(u)
+\frac{\alpha N}{p+1} 
\int_{\Om} \int_0^{u(x)} \frac{a(s)}{\ln(e+s)}\frac{s}{e+s}\,ds\,\,dx \nonumber\\
&\qquad +
\frac{\beta N}{q+1} \int_{\Om}\int^{v(x)}_0 \frac{b(s)}{\ln(e+s)}\frac{s}{e+s}\,ds \,\,dx
=  \int_{\p\Om}\frac{\partial u}{\p \nu}\frac{\partial v}{\p \nu}\, (x\cdot \nu).
\end{align} 

Introducing \eqref{h:crit} in \eqref{PRM:4a}, we deduce
\begin{align}\label{PRM:3}
&\frac{\alpha N}{p+1} 
\int_{\Om} \int_0^{u(x)} \frac{a(s)}{\ln(e+s)}\frac{s}{e+s}\,ds\,\,dx \\
&\qquad \nonumber +
\frac{\beta N}{q+1} \int_{\Om}\int^{v(x)}_0 \frac{b(s)}{\ln(e+s)}\frac{s}{e+s}\,ds \,\,dx
=  \int_{\p\Om}\frac{\partial u}{\p \nu}\frac{\partial v}{\p \nu}\, (x\cdot \nu).
\end{align}
Assume $\al,\be\le 0$,  if   $u\not\equiv 0$, then $v\not\equiv 0$, and 	the l.h.s. of \eqref{PRM:3} is non-positive. On the other hand, if $\Om$ is strongly star-shaped  with respect to $0$, the r.h.s. of \eqref{PRM:3} is strictly positive, giving a contradiction, and necessarily, $u\equiv 0$, and $v\equiv 0$.

Assume  $u\not\equiv 0$, then $v\not\equiv 0$, and introducing \eqref{h:crit:2} in \eqref{PRM:4a}, we deduce
\begin{align}\label{PRM:4}
&\frac{\alpha N}{p+1} 
\int_{\Om} \int_0^{u(x)} \frac{a(s)}{\ln(e+s)}\frac{s}{e+s}\,ds\,\,dx \\
&\qquad \nonumber +
\frac{\beta N}{q+1} \int_{\Om}\int^{v(x)}_0 \frac{b(s)}{\ln(e+s)}\frac{s}{e+s}\,ds \,\,dx
>  \int_{\p\Om}\frac{\partial u}{\p \nu}\frac{\partial v}{\p \nu}\, (x\cdot \nu).
\end{align}
Assume $\al,\be\le 0$, then the l.h.s. of \eqref{PRM:4} is non-positive. On the other hand, if $\Om$ is  star-shaped  with respect to $0$, the r.h.s. of \eqref{PRM:4} is strictly positive, giving a contradiction, and necessarily, $u\equiv 0$, and $v\equiv 0$.

\end{proof}

We point out that,
in general if $p$ and $q$ lie on the critical hyperbola,
  it would be interesting to check if the condition
$$ \frac{\al}{p+1}+ \frac{\be}{q+1}\le0$$
ensures  that
\begin{align*}
&\alpha \left(
\int_{\Om} \left[\int_0^{u(x)} \frac{a(s)}{\ln(e+s)}\frac{s}{e+s}\,ds - \int^{v(x)}_0 \frac{b(s)}{\ln(e+s)}\frac{s}{e+s}\,ds\right]\,dx \right)
\le 0  
\end{align*}
or equivalently
$$
N \int_{\Om} \left[A(u)+B(v)\right] -{(N-2)} \int_{\Om}\,\left[\theta\,u\,a(u)+(1-\theta)\,v\,b(v)\right] \le 0.
$$

\section{On Orlicz spaces}
\label{sec:orl}
Let us sumarise in this section some basics results  on Orlicz spaces (cf.  (\cite{Adams},\cite{Krasnoselski-Ruticki} and \cite{Rao-Ren}).
\subsection{${\mathcal N}$-functions}$ $

\begin{defi}(${\mathcal N}$-functions) $ $ \label{Nfunctions}
A function  $H:[0,\infty)\to[0,\infty)$  is said to be a 
${\mathcal N}$-function  if an only if $H$ is continuous, convex, 
$ H(t) = 0 $ if and only if $t = 0$ and 
$$\lim_{t\to 0}H(t)/t=0, \quad 
\lim_{t\to+\infty}H(t)/t=+\infty.$$
\end{defi}
Any ${\mathcal N}$-function is extended to $\mathbb{R}$ as an even function.  We denote  by $h$ the left derivative of $H$.

\subsection{Orlicz classes, Orlicz spaces and  the Luxembourg norm} $ $

Associated to the ${\mathcal N}$-function $H$ we have the following class of functions. 
Let $\Om\subset \mathbb{R}^N$ be here  an arbitrary  open set.

For our application, $\Om\subset \mathbb{R}^N$ is a {\bf bounded} set.

\begin{defi} (Orlicz class) $ $
The {\it  Orlicz class $K^H (\Om)$} is defined by
$$K^H (\Om) := \left\{u:\Om\to\mathbb{R} : u
\hbox{ is  measurable and }
\int_\Om H\big(u(x)\big)\, dx<+\infty\right\}.$$
\end{defi}

Orlicz classes are convex sets, but in general not linear spaces.

\begin{defi} (Orlicz space). \label{LH}$ $
We say that $u\in L^H (\Om)$ if and only if there exists a constant $c>0$ such that $cu\in K^H (\Om).$
$L^H (\Om)$  is a vector space, and it 
is called the 
{\it Orlicz space associated to $H$}.
\end{defi}

Let us recall the following result: 
\begin{lem}\label{delta2bis}
Let $H$ be an ${\mathcal N}$-function. Then, 
\begin{enumerate}
\item[\rm(i)] $L^H(\Om)=K^H(\Om)$ if and only if $H$ satisfies the so called {\bf $\Delta_2$-condition}:
\begin{equation*}
\exists  k\geq 1, \quad   H(2t)\leq kH(t)  \quad \forall t\geq 0.
\end{equation*}
\item[\rm (ii)] If $\Om$ is bounded then $L^H(\Om)=K^H(\Om)$ if and only if $H$ satisfies the so called {\bf $\Delta_2$-condition  at infinity} :
\begin{equation}\label{d2infty}
\exists  k\geq 1, \, \exists t_0\geq 0  ,\quad H(2t)\leq kH(t)  \quad \forall t\geq t_0.\end{equation}
\item[\rm (iii)]  Assume $H$ is derivable and its derivative $h$ is continuous and  strictly increasing. If 
\begin{equation}\label{delta22}\lim_{t\to+\infty}\frac{th(t)}{H(t)}=\alpha\in (1,+\infty)\end{equation}
then both $H$ and $\widetilde{H}$ satisfy the $\Delta_2$-condition at infinity,  where $\widetilde{H}$ is defined below in Definition \ref{def:Y:conj}.\footnote{\cite[Theorem 4.2]{Krasnoselski-Ruticki} clarify that $H$ and $\widetilde{H}$ satisfies the $\Delta_2$-condition if and only if $H$  satisfies the $\Delta_2$-condition and the  $\na_2$-condition, in other words, there exist numbers $h>1$ and $t_1 \ge 0$ such that
$$
H(t)\le \frac1{2h}H(ht)
\quad\text{for}\ t \ge t_1.
$$
}
\end{enumerate}
\end{lem}
\begin{proof}
(i) It follows from \cite[Theorem 8.2]{Krasnoselski-Ruticki},  {\rm (ii)}  from \cite[Theorem 2]{Rao-Ren} and {\rm (iii)} from \cite[Theorem 4.1 and Theorem 4.3]{Krasnoselski-Ruticki}.  
\end{proof}

\bigskip

We can provide the Orlicz space $L^H(\Om)$  with the {\it Luxemburg norm} : 
\begin{equation}
\label{def:L-norm}
\|u\|_{(H)}:=\inf\left\{\lambda>0\ :\ \int\limits_\Om H\left(\frac{u(x)}{\lambda}\right)\, dx\le 1\right\}.
\end{equation}

\begin{lem}\label{Ban}
\begin{enumerate}
\item[\rm(i)] $\left(L^H(\Om), \|\cdot\|_{(H)}\right)$ is a Banach space.  
\item[\rm (ii)] If $u\in L^H (\Omega)$, $u\not\equiv 0$,  then $\int_\Om H\left({\frac{u}{\|u\|_{(H)}}}\right)\, dx\le1$.
\item[\rm (iii)]  $\|u\|_{(H)}\leq \max\{\int_\Om H(u)\, dx,\, 1\}$.
\item[\rm (iv)] If $H$ satisfies the $\Delta_2$-condition then,  for all $u\not=0$, $u\in L^H (\Omega)$, 
it holds $$\int_\Om H\left({\frac{u}{\|u\|_{(H)}}}\right)\, dx=1.$$
If $\Om$ is bounded, the conclusion holds if $H$ satisfies the $\Delta_2$- condition at infinity.
\end{enumerate}
\end{lem}
\proof 
The proof of (i)   can be found in \cite{Rao-Ren}[Theorem 10, p. 67]. The proof of {\rm (ii)} and {\rm (iii)} follow trivially form the definition of the Luxembourg norm.
The result of {\rm (iv)} is proved in \cite{Rao-Ren}[Proposition 6, p. 77].
\qed

\medskip

\subsection{The Young-conjugate of a ${\mathcal N}$-function} 
\begin{defi}\label{def:Y:conj}
Let $H$ be an $\mathcal N$-function and denote  $\widetilde{H}$ its  Legendre transform, i.e.
$$\widetilde{H}(s):=\sup\limits_{\sigma\in\mathbb R}\{s\sigma-H(\sigma)\}.$$
Then, $\widetilde{H}$ is called the {\it Young-conjugate of $H$}.
\end{defi}

Some authors call $\widetilde{H}$ the complementary function of $H$, see \cite[p. 11-13]{Krasnoselski-Ruticki}. We denote by $\tilde{h}$ the   right derivative of $\widetilde{H}$. 

\begin{rem}\label{tildeh:h-1}
If $h$ is strictly increasing, then $\tilde{h}(t)=h^{-1}(t)$.
\end{rem}

The following properties are trivial to prove.
\begin{pro}\label{pro:H*}$ $
Let $H$ be an $\mathcal N$-function.
\begin{enumerate}
\item[\rm(i)] $\widetilde{(\widetilde{H})}=H$.
\item[\rm (ii)]  $\widetilde{H}$ is an ${\mathcal N}$-function. 
\item[\rm (iii)]  $H$ and $\widetilde{H}$ satisfy the {\it Young inequality} :
\begin{equation*}
\forall s,t \in (0, +\infty),\quad st\leq H(s)+ \widetilde{H}(t)
\end{equation*}
and the equality holds if and only if $t = h(s)$ or 
$ s = \tilde{h} (t).$ 
\item[\rm (iv)] The following {\it  1st Hölder's inequality } holds :
\begin{equation*}
\forall  f\in L^H(\Om), \, \forall g\in L^{\tilde{H}}(\Om),\quad 
\int\limits_\Om|f(x)g(x)|\, dx\le  2\|f\|_{(H)}\|g\|_{(\tilde{H})}.
\end{equation*}
\end{enumerate}
\end{pro}
\begin{proof}
The proof of (i)-(ii) can be found for instance in \cite{Krasnoselski-Ruticki}[Chapter 1].
For the proof of (iv), see \cite[Proposition 1 and Remark in p. 58]{Rao-Ren}.
\end{proof}

\subsection{Dual norm, the dual of an Orlicz space and reflexivity}$ $
\begin{defi}
Let $H$ be an $\mathcal N$-function. For all $u\in L^{H}(\Om)$ we define {\it the dual norm of $u$  or Orlicz norm of $u$} as 
\begin{equation}\label{onorm}
\|u\|_H :=\sup\left\{\int_\Om uv \, dx :\, \|v\|_{(\tilde{H})}\leq 1\right\}.
\end{equation}
\end{defi}
We have 
\begin{pro}
\label{proLH}

\begin{enumerate}
\item[\rm(i)] For all  $u\in L^{H}(\Om)$ 
$$\| u\|_{(H)}\leq \|u\|_H\leq 2\|u\|_{(H)}.$$
\item[\rm (ii)] 2nd Hölder's inequality :
\begin{equation}\label{2HIn}
\forall  u\in L^H(\Om), \, \forall v\in L^{\tilde{H}}(\Om),\quad 
\int\limits_\Om |u(x)v(x)| \, dx\le  \|u\|_{H} \|v\|_{(\tilde{H})}.\end{equation}
\item[\rm (iii)]For all  $u\in L^{H}(\Om)$ , $u\not=0$, 
\begin{equation}\label{2HIn2} \int_\Om H\bigg( \frac{u}{\|u\|_H} \bigg)dx \leq 1.\end{equation}
\item[\rm (iv)]  If $\widetilde{H}$ satisfies the $\Delta_2$-condition then the dual space $\big(L^H(\Om), \|\cdot \|_{H}\big)'$ coincides with $\big(L^{\tilde{H}}(\Om), \|\cdot \|_{(\tilde{H})}\big)$. 

\item[\rm(v)]   $L^H(\Om)$ is a reflexive Banach space if and only if $H$  and $\widetilde{H}$ satisfy 
the $\Delta_2$-condition.
\end{enumerate}
\medskip
If $\Om$ is bounded, the conclusions {\rm (iv)} and {\rm (v)}  hold if $H$  and $\tilde H$ satisfy the $\Delta_2$- condition at infinity.
\end{pro}
\proof
The proof of (i), (ii), (iii) follows from \cite[Lemma 9.2  and Theorem 9.3, p. 74]{Krasnoselski-Ruticki}.
For the proof of (iv) and (v) see \cite{Rao-Ren}[Theorem 10, p. 113].
\qed

\begin{rem}\label{remLH}
Since Proposition \ref{proLH} {\rm (ii)} and (i), 2nd Hölder's inequality can be written:
\begin{equation*}
\forall  u\in L^H(\Om), \, \forall v\in L^{\tilde{H}}(\Om),\quad 
\int\limits_\Om |u(x)v(x)| \, dx\le  \|u\|_{H} \|v\|_{\tilde{H}}.
\end{equation*}
Observe that Proposition \ref{proLH} {\rm (iv)}  is a Riesz representation theorem for Orlicz spaces, whenever $\widetilde{H}$ satisfies the $\Delta_2$-condition.
\end{rem}

\begin{defi}
We say that the sequence $\{u_{n}\}_{n\in \N}$ {\it converges in $H$-mean} to $u$ whenever 
\begin{equation*}
\lim_{n\rightarrow \infty}\int_\Om H\Big(|u_{n}(x)-u(x)|\Big)\, dx  =0.
\end{equation*}
\end{defi}
The folowing Theorem states that convergence in $H$-mean is equivalent to the convergence with respect to the Orlicz-norm, provided  that the $\Delta_2$-condition is satisfied, see \cite[Theorem 9.4]{Krasnoselski-Ruticki}.

\begin{thm}\label{th:conv:mean=norm} 
Let  $H$ be an $\mathcal N$-function satisfying the $\Delta_2$-condition.
Then, the convergence in $H$-mean is equivalent to the convergence with respect to the $\|\cdot\|_H$ norm.
\end{thm}

\subsection{Comparison of ${\mathcal N}$-functions}
It is possible to consider different partial ordering relations between  ${\mathcal N}$-functions, and they imply continuous embedding into Orlicz spaces.

First let us introduce the following ordering relations.
The following partial ordering relation between functions is involved in embedding theorems between Orlicz space associated with different ${\mathcal N}$-functions: 

\begin{defi}\label{def:dom} Let $H$ and $H_1$ be two ${\mathcal N}$-functions.
\begin{enumerate}
\item[\rm(i)] The function $H_1$ is said to {\it dominate the function} $H$ globally (respectively near infinity), denoted by $H  \prec  H_1,$ if there exists a positive constant $c$  such that 
\begin{equation*}
H(s) \le  H_1(cs),\qq{for} s \ge 0\qquad (s\ge s_0).
\end{equation*} 

\item[\rm (ii)] The functions $H$ and $H_1$ are called {\it equivalent globally} (near infinity), denoted by $H\sim H_1,$ if each dominates the other globally (near infinity). 

\item[\rm (iii)] If for every $c > 0$, there exists  a number $s_c\ge 0$ such that 
\begin{equation*}
H(s) \le  H_1(cs),\qq{for} s \ge 0\qquad (s\ge s_c),
\end{equation*} 
then $H$ is said to {\it increase essentially more slowly} than $H_1$ (at infinity), and in this case we write $H  \prec\hspace{-1.5mm}\prec  H_1$. 
\end{enumerate} 
\end{defi}

\begin{rem}\label{def:ll} There are different partial ordering relations between  ${\mathcal N}$-functions.

\begin{enumerate}
\item[\rm(i)] It is said that $H$ {\it growths  more slowly} than $H_1$ at infinity, and it is written  $H\ll H_1$, if and only if 
$$\forall r>0,\quad \lim_{s\to+\infty} \frac{H(rs)}{H_1 (s)}=0.$$
In \cite[p. 15, Definition 1, (ii)]{Rao-Ren}  is said that $H_1$ is   {\it essentially stronger} than $H$. This definition is equivalent to \ref{def:ll} {\rm (i). See \cite[p. 16, Theorem 2.b, (i) $\iff$ (v)]{Rao-Ren}}.

\item[\rm (ii)] In particular, if 
$$\exists c>0, \quad  \lim_{s\to+\infty} \frac{H(s)}{H_1 (s)}=c,$$
then $H$ is {\it equivalent} to $H_1$ at infinity.

\end{enumerate}
\end{rem}

\begin{pro}\label{pro:ll}
Let $H$ and $H_1$ be two ${\mathcal N}$-functions. 

The continuous embedding $L^{H_1}(\Om ) \to L^H(\Om )$ holds if either $H_1$ dominates $H$ globally, or $|\Om| <  \infty$  and $H_1$ dominates $H$ near infinity.

Moreover, 
$L^{H}(\Om)\subset L^{H_1} (\Om)$  if and only if there exists a constant $C>0$ such that $\|u\|_{H_1}\le C\|u\|_{H}$ for all $u\in L^{H}(\Om).$
\end{pro}

\begin{proof}
The proof of the first statement, follows directly from definition \ref{LH} of Orlicz space, and from definition \ref{def:dom}(i).

For the  proof of the second  estatement, see  \cite[Theorem 4, p. 51]{Luxemburg}. 
\end{proof}

\begin{pro}\label{pro:iso}
Let $H$ and $H_1$ be two ${\mathcal N}$-functions. If $H\sim H_1$ then $L^H (\Om)$ is isomorphic to $L^{H_1}(\Om)$.
\end{pro}

\begin{example}[Some useful facts]\label{example1}
Let $p>0$, $\alpha<p$, and 
$$a(t):=\frac{t^p}{(\ln (e+t))^\alpha},\ t\geq 0.$$
Then, we have
\begin{enumerate}
\item[\rm(i)] 
$ A(s)\sim s^{p+1}(\ln s)^{-\alpha}$,\\[-1.5mm]

\item[\rm (ii)] $\tilde{a}(s):=a^{-1}(s)\sim {s^\frac1p (\ln s)^\frac\alpha p}$,\\[-1.5mm]

\item[\rm (iii)] $\widetilde A(s)\sim {s^{\frac{p +1}p}(\ln s)^\frac\alpha p}$,\\[-1.5mm]

\item[\rm (iv)] 
$\widetilde A^{\, -1}(s)\sim  s^{\frac p{p +1}}(\ln s)^{-\frac\alpha {p+1}}$. 
\end{enumerate}
See Lemma \ref{lem:tilde:aA}.
\end{example}

\subsection{The  Orlicz-Sobolev spaces   $W_0^{1, H} (\Om)$  and $W^{m,H} (\Om)$.  
}
Let $H$ be an $\mathcal N$-function and let $m\in \mathbb{N}^*$.
\begin{defi}\label{def:O-S}
\begin{enumerate}
\item[\rm(i)] The space  $W^{m,H} (\Om)$ is defined as 
$$W^{m,H} (\Om):=\big\{u:\Om\to \mathbb{R}\, :   D^\alpha u \in L^H (\Om) \quad \forall \, |\alpha|\in \{0,1,, \dots, m\}\big\}$$
where $D^\alpha  $ stands for the weak partial derivative of $u$. 
\item[\rm (ii)] The Luxemburg norm for the elements $u$ of this space is 
$$\|u\|_{m,(H) }:= \max_{0\leq |\alpha |\leq m}\|D^{\alpha} u\|_{(H)}$$
and $\big( W^{m,H} (\Om), \|\cdot\|_{m,(H) }\big)$  is a Banach space.
\item[\rm (iii)] The space  $W_0^{1,H} (\Om)$ is defined as 
$$W_0^{1,H} (\Om):=\overline{C_0^\infty(\Om)}^{\|\cdot\|_{1,(H )}}.$$
\item[\rm (iv)]   The Orlicz norm for $u  \in W_0^{1,H} (\Om)$ is defined by 
$$\|u\|_{1,H}:= \sup\left\{\int_\Om uv\, \, dx \, : v\in W_0^{1,H} (\Om), \;  \|v\|_{1,(H)}\leq 1\right\}.$$
The Orlicz norm is equivalent to the Luxemburg norm, see Proposition \ref{proLH}(i).
\end{enumerate}
\end{defi}
\medskip

\subsection{Optimal Embedding Theorems for Orlicz-Sobolev spaces}

We present in this section some known results concerning 
the  {\it optimal Orlicz--Sobolev  embedding}.  For that purpose,  we define for any   any ${\mathcal N}$-function $H$, the {\it auxiliary  function} 
\begin{equation}\label{def:HN:PhiN}
\Phi_H (s):=\int_0^s  \frac{\widetilde{H}(\tau)}{\tau^{1+N'}} \,d\tau
\end{equation} 
where   $N'=\frac{N}{N-1}$.   We will denote by $\to$ a continuous  embedding, and by $\hookto$ a compact embedding.
\begin{rem}\label{rem:H:0}
We can always assume that $\Phi_H$ is well defined. 
Indeed,  since $\Om$ is of finite measure (it is  bounded), we can 
assume without loss of generality that 
$$\int_0 \frac{\tilde H (\tau)}{\tau^{1+N'}} d\tau <\infty$$  otherwise $H$ can be replaced by any ${\mathcal N}$-function which is equivalent to
the original one near infinity and makes the previous integral converge. Such
a replacement does not affect the result since the corresponding Orlicz-Sobolev
norm is equivalent to the original one. 
\end{rem}
It turns out that one finds different  optimal embeddings depending on whatever 
\begin{equation}\label{iH}i_H:=\int^\infty \frac{\widetilde{H}(\tau)}{\tau^{1+N'}} d\tau\end{equation}
is finite or not.  Notice that if $i_H=+\infty$, in particular $\Phi_H$ is strictly increasing (at  least for $s$ large) and $\Phi_H^{-1} $ is well defined.
Consequently we distinguish two cases :\\

\noindent {\bf Case 1 : $i_H<+\infty$.} The following Theorem is a compact Orlicz--Sobolev  embedding for the spaces $W_0^{1,H}(\Om )$ and $W^{1,H}(\Om)$, in terms of continuous bounded functions, see \cite{Donaldson-Trudinger}, and \cite[Corollary 1]{Cianchi_ASNSP_1996}.  Recall that a bounded open set $\Om$ is called {\it strongly Lipschitz} if,
for each $x \in \p\Om$, there exist a neighbourhood $U_x$ of $x$, a coordinate system
$(y_1, ... , y_N)$ centered at $x$ and a Lipschitz continuous function $\phi$ of $(y_1, . . . , y_{N-1})$ such that
\begin{equation*}
\Om\cap U_x=\{(y_1, ... , y_N):\ y_N>\phi(y_1, . . . , y_{N-1})\}.    
\end{equation*}

\bigskip

\begin{thm}\label{th:cont:emb:cont:comp}
Let $\Om\subset  \R^N$, $N\ge 2$ be a bounded open set. 
Let $H$ be an $\mathcal N$-function   satisfying $i_H<+\infty$, where $i_H$ is defined in \eqref{iH}.   Let  $C_b(\Om ) $ be  the space of continuous bounded functions on $\Om$. 

\begin{enumerate}
\item Then, the embedding
$$	
W^{1,H}_0(\Om) \hookto C_b(\Om ) 
$$ 
is compact.
\item If in addition $\Om$ has the strong Lipschitz property then,   the embedding 
$$
W^{1,H}(\Om) \hookto C_b(\Om )
$$
is compact. 

\end{enumerate}
\end{thm}

\noindent {\bf Case 2 : $i_H=+\infty$}. 
The following theorem gives the {\it optimal Orlicz--Sobolev  
continuous embedding} for the space $W^{1,H}_0(\Om )$, in terms of Orlicz spaces, see \cite[Theorem 1]{Cianchi_IUMJ_1996}.\\

\begin{thm}\label{th:cont:emb:0}
Let $N\ge 2$. Let $H$ be any ${\mathcal N}$-function and let  $\Phi_H$ be the auxiliary function defined in \eqref{def:HN:PhiN}
Assume that $i_H=+\infty$, where $i_H$ is defined in \eqref{iH}. Define for all $s\geq 0$
\begin{equation}\label{def:HN}
H^*(s) =H_N^*(s) = \int_0^s t^{N'-1}\left(\Phi_H ^{-1}\big(t^{N'}\big)\right)^{N'} dt 
\end{equation}   
Then $H^*$  is an ${\mathcal N}$-function,  and the following continuous embedding holds:
\begin{equation}\label{ineq:cont:emb:0}
W_0^{1,H} (\Om)\to L^{H^*} (\Om).
\end{equation}
Furthermore $L^{H^*}(\Om)$ is the smallest Orlicz space that renders \eqref{ineq:cont:emb:0} true.
\end{thm}

\bigskip

The following theorem is a also a continuous Orlicz--Sobolev embedding, this time for the space $W^{1,H}(\Om )$, in terms of Orlicz spaces, see \cite[Theorem 2]{Cianchi_IUMJ_1996}. We will say that	$\Om\subset \R^N$  satisfies the {\it cone property}, if there exists a cone $\Sigma$ such that for any $x\in\Om$, the set $\Om$ contains a cone congruent with $\Sigma$ and whose vertex is $x$.

\begin{thm}\label{th:cont:emb}
Let $N\ge 2$.  Assume that  $\Om\subset \R^N$ is any open bounded  connected set, satisfying the cone property.
Let $H$ be any $\mathcal N$-function, and let $H^*$ be the function defined by \eqref{def:HN}.  Assume that $i_H=+\infty$, where $i_H$ is defined in \eqref{iH}.
Then, the following holds:

\begin{enumerate}
\item[\rm (i)]  There exists a constant $C$, depending only on $H$, $|\Om|$ and $N$,  such that  
\begin{equation*}
\|u -u_\Om\|_{L^{H^*}(\Om)} \le C\|\nabla u\|_{(H)}\qquad \text{for all } u \in W^{1,H}(\Om). 
\end{equation*}
Here, 
$$
u_\Om := \frac{1}{|\Om|} \int_\Om u(x)\, \, dx
$$ 
is the mean value of $u$ over $\Om$.  	

\item[\rm (ii)] 	The continuous embedding 
\begin{equation*}
W^{1,H}(\Om) \to L^{H^*} (\Om )
\end{equation*}
holds, where possibly $H^*$ is replaced at zero, in the sens of Remark \ref{rem:H:0}.
\end{enumerate}
\end{thm}

\bigskip

Finally we give in  Theorem \ref{th:comp:emb} below a compact Orlicz--Sobolev  embedding for the space $W^{1,H}(\Om )$, this time in terms of Orlicz spaces, see \cite[Theorem 3]{Cianchi_IUMJ_1996}.		 

\bigskip

\begin{thm}\label{th:comp:emb}
Let $N\ge 2$ and let $\Om$ be any open, bounded, connected set, and satisfying the cone property. Let $H$ be any ${\mathcal N}$- function.   Assume that $i_H=+\infty$, where $i_H$ is defined in \eqref{iH} and   $H_1$ a ${\mathcal N}$- function increasing  essentially more slowly near infinity than $H^*$ defined by \eqref{def:HN}. Then the embedding 
\begin{equation*}
W^{1,H}(\Om)\hookto L^{H_1}(\Om) 
\end{equation*} 
is compact. 
\end{thm}
\bigskip 

\begin{example} Some useful continuous and compact embeddings \label{ex:HN:log}$ $

Consider any ${\mathcal N}$-function $H$ such that 
$$ H(s)\sim s^p \big[\log(s)\big]^\al$$ where either $
p > 1 $ and $\al\in\R $ or $p=1  $ and $ \al> 0.$

Let $\Om\subset\R^N$ be an bounded open set.
\medskip 

\noindent {\bf  Case 1. } 
The following compact  embedding holds :
$$
W^{1,H}_0(\Om) \hookto C_b(\Om )\qq{if  either}
\begin{cases}
p > N & \qq{and} \al\in\R,\\
\text{or}\quad p=N  & \qq{and} \al>N-1.
\end{cases}
$$

\noindent {\bf Case 2. } The following   continuous embedding holds :
$$
W^{1,H}_0(\Om)  \to L^{H^*}(\Om ) \qq{if} p\le N, 
$$ 
see Theorem \ref{th:cont:emb:0}, where 
\begin{equation}\label{def:HN:log}
H^* (s)\sim
\begin{cases}
\Big(s^{p} \big[\log(s)\big]^\al\Big)^{\frac{N}{N-p}} & \qq{if} 1 \le  p <  N,\\[2mm]
e^{s^{N/(N-1-\al)}} & \qq{if} p = N,\ \al<N-1,\\[2mm]
e^{e^{s^{N'}}} & \qq{if} p = N,\ \al=N-1. 
\end{cases}
\end{equation}

By Theorem \ref{th:cont:emb}, the same embeddings are true with $W^{1,H}(\Om)$ replacing $W^{1,H}_0(\Om)$, provided that $\Om\subset \R^N$ has finite measure and satisfies the cone property. 

\medskip

Moreover, if $p\le N,$ since Theorem \ref{th:comp:emb}
$$
W^{1,H}(\Om) \hookto L^{H_1}(\Om ) \qq{for any} H_1  \prec\hspace{-1.5mm}\prec  H^*.
$$
Notice that if $\al = 0$,  \eqref{def:HN:log} agrees with Sobolev's theorem  when $p\ne N$, and with Trudinger's theorem when $p= N$.
\end{example}

\begin{example}
Consider a ${\mathcal N}$-functions $\ H(s)\ $ which are equivalent near infinity to 
$$
s^p \Big(\log\big[\log(s)\big]\Big)^\al, \qq{where either} \begin{cases}
p > 1 & \text{ and } \al\in\R,\\
\text{or}\quad p=1  & \text{ and }  \al> 0. 
\end{cases}
$$
Then, Theorem \ref{th:cont:emb:0},  
and Theorem \ref{th:cont:emb:cont:comp} imply that 
$$
W^{1,H}(\Om) \hookto C_b(\Om ) \qq{if} 
p > N,
$$  
and 
$$
W^{1,H}(\Om) \to L^{H^*}(\Om ) \qq{if} p\le N,
$$ 
where
$$
H^* (s)\sim\begin{cases}
\Big(s^{p} \big[\log\log(s)\big]^\al\Big)^{\frac{N}{N-p}}  & \qq{if} 1 \le  p <  N, \\[3mm]
e^{\left(s^{N}\,[\log(s)]^\al \right)^{\frac{1}{N-1}}}\,  & \qq{if} p = N. 
\end{cases}
$$
\end{example}

\begin{thm}\label{th:emb}  Assume that  $\Om$ has the strong Lipschitz property.
Let $C_b^{\,1}(\Om)$ be the set of $C^1$ functions such that their derivatives are bounded in $L^\infty(\Om)$. Let $\tilde A$ the ${\mathcal N}$-function  defined in \eqref{def:tilde:A:B}.
The following Orlicz--Sobolev continuous embeddings hold :
\begin{equation}\label{OS:cont:emb} 
W^{2,{\tilde A}}(\Om)\subset 
\begin{cases}
L^{({\tilde A}^*)^*} (\Om)&\qq{if} 
\frac{p+1}{p}\le \frac{N}{2},\ 
\al\le \left(1-\frac{1}{N}\right)\big(Np-(p+1)\big)\\[2mm]
C_b (\Om)&\qq{if} \frac{p+1}{p}= \frac{N}{2},\ \al> \left(1-\frac{1}{N}\right)\big(Np-(p+1)\big),\\
&\quad \qq{or }  \frac{p+1}{p}> \frac{N}{2},\ \al\in\R;\\[2mm]
C_b^{\,1}(\Om)&\qq{if} \frac{p+1}{p}= N,\ \al> p(N-1),\quad\text{or } \frac{p+1}{p}> N,\ \al\in\R;
\end{cases}
\end{equation}
with $\big({\widetilde A}^*\big)^*$ defined by 
\begin{equation}\label{def:Atilde:N:N}
\big({\widetilde A}^*\big)^*\sim	\begin{cases}
s^\frac{N(p+1)}{Np-2(p+1)}\,  	\big[\log(s)\big]^\frac{\al N}{Np-2(p+1)} & \text{if } \frac{p+1}{p} <  \frac{N}{2}, \\[2mm]
e^{s^{N/\big(N-1-\al\frac{N}{Np-(p+1)}\big)}} & \text{if } \frac{p+1}{p} = \frac{N}{2},\ \al<\left(1-\frac{1}{N}\right)\big(Np-(p+1)\big)\\[3mm]
e^{e^{s^{N'}}} & \text{if } \frac{p+1}{p} = \frac{N}{2},\ \al=\left(1-\frac{1}{N}\right)\big(Np-(p+1)\big). 
\end{cases}
\end{equation}
Moreover,  the following Orlicz--Sobolev compact embeddings hold
\begin{equation}\label{OS:emb:comp} 
W^{2,{\tilde A}}(\Om)\hookto 
\begin{cases}
L^{A_1} (\Om)&\qquad \forall A_1\prec\hspace{-1.5mm}\prec\big({\widetilde A}^*\big)^*, \qq{if} \frac{p+1}{p}\le \frac{N}{2},\\ 
&\quad \qq{and}  \al\le\left(1-\frac{1}{N}\right)\big(Np-(p+1)\big);\\[2mm]
C_b (\Om)&\qq{if} \frac{p+1}{p}= \frac{N}{2},\ \al> \left(1-\frac{1}{N}\right)\big(Np-(p+1)\big),\\
&\quad \qq{or }  \frac{p+1}{p}> \frac{N}{2},\ \al\in\R;\\[2mm]
C_b^{\,1}(\Om)&\qq{if} \frac{p+1}{p}= N,\ \al> p(N-1),\quad\text{or } \frac{p+1}{p}> N,\ \al\in\R;
\end{cases}
\end{equation}
for any $A_1$ increasing essentially more slowly than $\big(({\widetilde A})^*\big)^*$, denoted by $A_1\prec\hspace{-1.5mm}\prec\big(({\widetilde A})^*\big)^*$, see definition \ref{def:dom}{\rm (iii)}.
\end{thm}

\begin{proof}[Proof of Theorem \ref{th:emb}]
Since  Lemma \ref{lem:tilde:aA} and  Example \ref{example1}, 
$$   
\widetilde A(s)\sim s^{\frac{p+1}{p}}(\ln s)^{\frac{\al}{p}}.
$$
Moreover, by    Orlicz--Sobolev continuous embeddings, 
\begin{equation*}
W^{2,{\tilde A}}(\Om)\to \begin{cases}
W^{1,({\tilde A})^*} (\Om)&\qq{if} \frac{p+1}{p}\le N,\quad \al\le p(N-1)\\[2mm]
C_b^{\,1}(\Om)&\qq{if} \frac{p+1}{p}= N,\ \al> p(N-1),\quad\text{or } \frac{p+1}{p}> N,\ \al\in\R;
\end{cases}
\end{equation*}
with $({\widetilde A})^*$ defined  specifically by
\begin{equation*}
({\widetilde A})^*\sim	\begin{cases}
\Big(s^\frac{p+1}{p} \big[\log(s)\big]^\frac{\al }{p}\Big)^{\frac{N}{N-\frac{p+1}{p}}},
& \qq{if} \frac{p+1}{p} <  N, \\[2mm]
e^{s^{N/(N-1-\frac{\al}{p})}}, & \qq{if} \frac{p+1}{p} = N,\ \al<p(N-1)\\
e^{e^{s^{N'}}},& \qq{if} \frac{p+1}{p} = N,\ \al=p(N-1), 
\end{cases}
\end{equation*}
cf.   Theorem \ref{th:cont:emb},  Theorem \ref{th:cont:emb:cont:comp}, example \ref{ex:HN:log}, and definition \eqref{def:HN:log}. Observe that $\frac{N}{N-\frac{p+1}{p}}\frac{p+1}{p}
=\frac{N(p+1)}{Np-(p+1)}=\left(\frac p{p+1}-\frac{1}{N}\right)^{-1}$.

Iterating the above procedure we obtain \eqref{OS:cont:emb}, with $\big({\widetilde A}^*\big)^*$ defined,  by 
\begin{equation*}
\big({\widetilde A}^*\big)^*\sim\begin{cases}
\Big(s^\frac{N(p+1)}{Np-(p+1)} \big[\log(s)\big]^\frac{\al N}{Np-(p+1)}\Big)^{\frac{Np-(p+1)}{Np-2(p+1)}} & \text{if } \frac{p+1}{p} <  \frac{N}{2}, \\[2mm]
e^{s^{N/\big(N-1-\al\frac{N}{Np-(p+1)}\big)}} & \text{if } \frac{p+1}{p} = \frac{N}{2},\ \al<\left(1-\frac{1}{N}\right)\big(Np-(p+1)\big)\\[3mm]
e^{e^{s^{N'}}} & \text{if } \frac{p+1}{p} = \frac{N}{2},\ \al=\left(1-\frac{1}{N}\right)\big(Np-(p+1)\big), 
\end{cases}
\end{equation*}
which is obtained iterating twice \eqref{def:HN:log}. 
It can be equivalently rewritten as \eqref{def:Atilde:N:N}. Observe that 
$\frac{N(p+1)}{Np-(p+1)}\frac{Np-(p+1)}{Np-2(p+1)}
=\frac{N(p+1)}{Np-2(p+1)}=\left(\frac p{p+1} -\frac{2}{N}\right)^{-1}$.

Moreover, from Theorem \ref{th:comp:emb}, the  Orlicz--Sobolev compact embeddings described in \eqref{OS:emb:comp} hold, 
see definition \ref{def:dom}{\rm (iii)} for $A_1$ increasing essentially more slowly than $\big(({\widetilde A})^*\big)^*$, denoted by $A_1\prec\hspace{-1.5mm}\prec\big(({\widetilde A})^*\big)^*$.
\end{proof}

\begin{cor}\label{cor:tildeA:q+1}
In particular,  under the conditions of Theorem  \ref{th:emb}, assume that one of the following two conditions holds:

{\rm (i)} either \eqref{h:sub} is satisfied, for any $\al\in \R$

{\rm (ii)} either \eqref{h} is satisfied, and  $\al>0$.

Then, the following Orlicz--Sobolev compact embedding holds
\begin{equation}\label{OS:emb:q+1} 
W^{2,{\tilde A}}(\Om)\hookto L^{q+1}(\Om).
\end{equation}
If \eqref{h} is satisfied, and  $\al=0,$ then the above embedding is continuous.
\end{cor}
\begin{proof}
Using Theorem \ref{th:emb}, we only have to realize that $A_1(s)=s^{q+1}$
increases essentially more slowly than $\big({\widetilde A}^*\big)^*$, i.e.,  $A_1\prec\hspace{-1.5mm}\prec\big({\widetilde A}^*\big)^*$, see definition \ref{def:dom}{\rm (iii)}.

We first check that $L^{({\tilde A}^*)^*} (\Om) \subset  L^{q+1}(\Om)$. Indeed, 
$$
\frac{N(p+1)}{Np-2(p+1)}\ge q+1\iff\frac p{p+1}-\frac{2}{N}\le \frac1{q+1}\iff\frac{N-2}{N}\le \frac 1{p+1}+\frac1{q+1}.
$$

Moreover, since $\al>0$ in case (ii), then  $A_1\prec\hspace{-1.5mm}\prec\big(({\widetilde A})^*\big)^*$,  and so \eqref{OS:emb:q+1} holds. 

If \eqref{h} is satisfied, and  $\al=0,$ then $A_1=\big({\widetilde A}^*\big)^*,$ and the embedding is continuous.
\end{proof}

\subsection{The Dirichlet problem} \label{regularity}

Let $\Om\subset\mathbb{R}^N$ be a bounded domain of class $C^2$, $H$ a ${\mathcal N}$-function and $f\in L^H(\Om)$.  Let us consider the Dirichlet problem
\begin{equation}\label{dir}
-\Delta u=f\ \hbox{in}\ \Om,\qquad u=0\ \hbox{on}\ \partial\Om.
\end{equation}
Solutions are understood in the weak sense, i.e.
$$\forall \varphi \in C_0^1(\Om), \quad 
\int_\Om \nabla u\cdot \nabla \varphi \, dx=\int_\Om f\varphi \, dx.$$

We have the following regularity result (cf. \cite[Theorem XI.8]{Benkirane},  and \cite[Theorem 4]{Jia-Li-Wang}).

\begin{thm}\label{regu}
Let $\Om\subset\mathbb{R}^N$ be a bounded domain of class $C^2$.
Assume that $H$ and $\widetilde{H}$ satisfies the $\Delta_2$-condition. Then, the unique solution  $u=K(f)$ of  problem \eqref{dir} belongs to $ W^{2,H}(\Om) \cap W_0^{1,H}(\Om)$, $-\Delta u=f$ a.e.  and 
\begin{equation}\label{est2A}
\|u\|_H+\sum_i \big\|\partial_{x_i} u\big\|_H +\sum_{ij}\big\|\partial^2_{x_ix_j}u\big\|_H\le C\|f\|_H
\end{equation}
for some $C=C(N,H,\Om)$. 
\end{thm}
See  Lemma \ref{delta2bis} (iii) for having a sufficient condition guarantying that $H$ and $\widetilde{H}$ satisfies the $\Delta_2$-condition.

\bibliographystyle{abbrv}
\bibliography{ref}

\end{document}